\documentclass[11pt, a4paper]{amsart}
\usepackage[foot]{amsaddr}
\usepackage{amssymb,amsmath,enumitem, bm,  eucal,  exscale,  mathrsfs,  ulem}
\usepackage{graphicx}
\usepackage{xurl}
\usepackage[pagewise]{lineno}
\usepackage[hyperindex,breaklinks]{hyperref}
\usepackage[dvipsnames]{xcolor}
\hypersetup{colorlinks=true, citecolor=NavyBlue, linkcolor=OliveGreen, urlcolor=Maroon}
\usepackage{graphicx}
\usepackage{comment}

\newtheorem{theorem}{Theorem}[section]
\newtheorem{lemma}{Lemma}[section]
\newtheorem{prop}{Proposition}[section]
\newtheorem{cor}{Corollary}[section]
\newtheorem{remark}{Remark}[section]
\newtheorem{defi}{Definition}[section]

\numberwithin{equation}{section}

\newcommand{\real}{{\mathbb R}}
\newcommand{\form}{\mathcal E}
\newcommand{\dom}{\mathcal F}

\newcommand{\vareps}{\varepsilon}
\newcommand{\dis}{\displaystyle}

\newcommand{\rd}{{\mathbb R}^d}

\makeatletter
\newcounter{con}

\makeatother

\allowdisplaybreaks[4] 
\title[Homogenization of diffusion processes]{Homogenization of diffusion processes with singular drifts and potentials via unfolding method}
\author{Toshihiro Uemura}
\address[Toshihiro Uemura]{Department of Mathematics, Faculty of Engineering Science, 
Kansai University, Suita, Osaka 564-8680, Japan} \email{\href{mailto:t-uemura@kansai-u.ac.jp}{t-uemura@kansai-u.ac.jp} (T. Uemura)}

\author{Adisak Seesanea}
\address[Adisak Seesanea]{Sirindhorn International Institute of Technology, Thammasat University, Pathum Thani 12120, Thailand}
\email[A. Seesanea]{\href{adisak.see@siit.tu.ac.th}{adisak.see@siit.tu.ac.th} (A. Seesanea)}
\begin{document}
\begin{abstract}
This work is concerned with homogenization problems for elliptic equations of the type
\[
\begin{cases}
\mathfrak{L}_{\delta} u_{\delta} + \lambda u_{\delta} = f_{\delta}  \qquad \text{in} \;\; D, \\
\qquad \qquad \; u = 0 \qquad \,\, \,  \text{on} \;\; \partial D,
\end{cases}
\]
where $\delta > 0$,  $\lambda \in \mathbb{R}$,  $D$ is a bounded open set in $\mathbb{R}^{d}$,  
and $f_{\delta} \in H^{-1}(D)$.   The operator  
$
\mathfrak{L}_{\delta} u = -{\rm div} \left( A^\delta  \nabla u  +  
u\,C^\delta \right) + B^\delta \nabla u +k^\delta u
$
with uniformly bounded coefficients $A^\delta$,  where drifts $B^\delta$,  $C^\delta$, and potential $k^\delta$ are possibly unbounded. An application to the homogenization of the corresponding diffusion processes is also discussed. 
\end{abstract}
\maketitle
\section{Introduction}\label{sect:intro}
Classical homogenization problems involve the study of Dirichlet problems for second-order elliptic partial differential equations with periodic diffusion coefficients of the form
\begin{equation} \label{classicalPDE}
\begin{cases}
\mathcal{L}_{\delta} u_{\delta} = f_{\delta}  \qquad \text{in} \;\; D, \\
\quad \;\, u = 0 \qquad \,  \text{on} \;\; \partial D,
\end{cases}
\end{equation}
where $\delta> 0$ is a small parameter that stands for the heterogeneity (or the length of periodic structure) of a 
body $D$ which is assumed to be a bounded open subset in the Euclidean space $\rd$  where $d \geq 1$,  and
the source $f_{\delta}$ belongs to  $H^{-1}(D)$,  i.e.,  the dual space of the Sobolev space $H^1_0(D)$
(see Section \ref{sec:notation} for definition).  

The operator 
$\mathcal{L}_{\delta} u  =  -\text{div}(A^{\delta} \nabla u)$ where 
 $A^{\delta}(x) = A(\frac{x}{\delta}) = (a_{ij}(\frac{x}{\delta}) )$ is a bounded $d \times d$ matrix-valued measurable function defined on $\rd$ 
such that each $a_{ij}$ is $Y$-periodic  in the sense that for $m = 1, 2, \ldots,d$ and $\ell \in \mathbb{Z}$, 
$$
a_{ij}(x + \ell {\bm e}_m) = a_{ij}(x)  \qquad {\rm a.e.}  \;\;   x \in Y, 
$$
where $Y=(0,1)^d$ is the unit cube in $\rd$ and ${\bm e}_m$ is the $m$-th directional 
unit vector of $\rd$.   Therefore,  in this case,  each  $A^{\delta}$ represents the diffusion coefficient inside the $\delta$ periodic structure.

The main goal of such homogenization problems is to find possible limit(s) 
$u_0$ to the sequence of solutions $\big(u_{\delta}\big)_{\delta>0}$  and determining the problem(s) 
for which the limit $u_0$ is the solution.

Observe that various approaches have been proposed to investigate the problems such as asymptotic expansions, 
$G$-convergence,  $2$-scale convergence,  $H$-convergence,  and unfolding method.  We refer the reader to
\cite{Al16,  BLP11,  CD10, CDG18,  Hor97,  LNW02} for a comprehensive discussion on corresponding methods. 

In the present work, we employ the unfolding method and techniques 
of semi-Dirichlet forms to study the homogenization of the Dirichlet problem
\begin{equation} \label{eqn:hom}
\begin{cases}
\mathfrak{L}_{\delta} u_{\delta} + \lambda u_{\delta} = f_{\delta}  \qquad \text{in} \;\; D, \\
\qquad \qquad \; u = 0 \qquad \,  \text{on} \;\; \partial D,
\end{cases}
\end{equation}
where $\delta > 0$,  $\lambda \in \mathbb{R}$,  $D \subset \mathbb{R}^{d}$ is a bounded open set, 
and $f_{\delta} \in H^{-1}(D)$.   
Here,  the general second-order differential operator
\[
\mathfrak{L}_{\delta} u = -{\rm div} \left( A^\delta  \nabla u  +  
u\,C^\delta \right) + B^\delta \nabla u +k^\delta u, 
\]
where $A^\delta=(a_{ij}^\delta)$ is a $d \times d$ matrix 
with $a_{ij}^\delta \in L^{\infty}(D)$,  and there exist positive constants $ \alpha \leq \beta $ so that 
\[
\alpha |\xi|^{2} \leq A^{\delta} (x)\xi \cdot \xi \leq \beta |\xi|^{2}
\]
for almost every $x \in D$ and every $\xi \in \mathbb{R}^d$. 
The drifts 
 $B^\delta =(b_i^\delta)$ and $C^\delta =(c_i^\delta)$ are arbitrary $d$-dimensional 
  vector-valued measurable functions on $\rd$,  and the potential $k^\delta$ is a measurable function on $\rd$. 

The Dirichlet problem \eqref{eqn:hom} will be understood in the weak sense
\[
\form^\delta(u_\delta, v) +\lambda (u_\delta, v)_{L^{2}} = \langle f, v \rangle_{H^{-1}, H^1_0}, \qquad v \in H^1_0(D), 
\]
where the quadratic form $\form^\delta(u,v)$ is given by
\[
\form^\delta(u,v) =
 \int_{D} A^\delta \nabla u \cdot \nabla v \, dx 
  + \int_{D} (B^{\delta} \cdot \nabla u )  v \,  dx 
+ \int_{D}  u  (C^{\delta} \cdot \nabla v) \,  dx
+ \int_{D} k^{\delta} \, u \, v \, dx.
\]

Under some additional mild assumptions on $A^{\delta}$ (not necessary of the single translation form $A^{\delta}(x) = A\left(\frac{x}{\delta} \right)$) and integrability
assumptions (possibly unbounded) on  $B^\delta$,  $C^\delta$ and $k^\delta$
imposed in Section \ref{S:semi-D},  we show that, for each $\delta > 0$,  the quadratic form
induces a regular lower bounded semi-Dirichlet form $(\form^{\delta}, H_0^1(D))$ on $L^2(D)$ with some 
(uniform) index $\beta_0>0$.  Therefore, there exists a diffusion 
process on $D$ associated with $(\form^{\delta}, H_0^1(D))$ (see \cite{O13}).  
As a consequence of our study of the homogenization of \eqref{eqn:hom},  
we can deduce the convergence of the diffusion processes in 
the finite dimensional distribution sense as $\delta$ tends to $0$. 

To sketch the idea of our approach to the homogenization problem, 
we first establish the existence and uniqueness of a solution 
$u^{\delta}_{\lambda}\in H^{1}_{0}(D)$ to \eqref{eqn:hom} for $\lambda>\beta_0$,  
where $\beta_0$ is the index of the semi-Dirichlet form $(\form^{\delta}, H_0^1(D))$.
Subsequently, we show that, if $f_{\delta}$ belongs to $L^2(D)$,  $u_{\lambda}^{\delta}=G_{\lambda}^{\delta}f$,  
where $\{G_{\lambda}^{\delta}\}_{\lambda>\beta_0}$ is the $L^2$-resolvent of $(\form^{\delta}, H_0^1(D))$.
Moreover, if in addition $f_{\delta}$ is convergent in $L^{2}(D)$, we deduce that $\{G_{\lambda}^{\delta}\}$ converges strongly  to some $L^2$-resolvent 
$\{G_{\lambda}^0\}$ 
in $L^2(D)$ as $\delta$ tends to $0$,  for which the limiting resolvent $\{G_{\lambda}^0\}$  corresponds to the semi-Dirichlet form $(\form^0, H^1_0(D))$
defined for $u,  v \in H^1_0(D),$ by
\begin{align*}
\form^0(u, v) &=\int_D A^{\sf eff} \, \nabla u \cdot \nabla v \, dx + 
\int_{D} B^{\sf eff} \cdot \nabla u \, v \, dx  +\int_D u \,  C^{\sf eff} \cdot \nabla  v \, dx 
 +  \int_{D} u \, v \, k^{\sf eff} \, dx  
\end{align*}
with some  {\it effective tensor} $A^{\sf eff} =(a_{ij}^{\sf eff})$,  {\it effective drifts} 
$B^{\sf eff} =(b_j^{\sf eff})$, $C^{\sf eff} =(c_j^{\sf eff})$ and 
{\it effective potential} $k^{\sf eff}$,  see Section \ref{S:hom}. 

To the extent of the authors' knowledge, our result is new since the drift terms $B^\delta$  and  $C^\delta$ are allowed to be unbounded  (see Section \ref{S:semi-D}).  
We note that earlier Jikov,  Kozlov and Oleinik \cite[pp. 31]{JKO94} adopted the multiscale expansion technique to investigate similar homogenization problems in the case $B^\delta$ and $C^\delta$ are bounded and of the class $C^2$,  see also \cite[Section 1.13]{BLP11} and \cite{WT23, X16}. 
Recently, Inagaki  considered a similar homogenization problem
of second-order elliptic differential operators with singular drifts (without potentials)
by using the semi-Dirichlet forms, see \cite{I19}.

The remainder of this paper is organized as follows.  In Section \ref{sect:dirichlet},  
we provide preliminary background on semi-Dirichlet forms,   construct  a sequence of closed forms 
 singular drifts and potentials,  and explore its properties. 
The concept of unfolding operators is discussed in  Section \ref{sec:unfold}. 
The statements, proofs, and consequences of our main result are established in Section \ref{S:hom}.
\section{Lower Bounded Semi-Dirichlet Form}\label{sect:dirichlet}
In this section, we first recall the notion of a lower bounded semi-Dirichlet form on some $L^2$-space 
briefly following \cite{MOR95} and \cite{O13}.

\subsection{Definition of a lower bounded semi-Dirichlet form}\label{sec:notation}

Let $E$ be a locally compact separable metric space and $m$ a positive Radon measure 
on $E$ with full support. Let $\dom$ be a dense subspace of $L^2(E):=L^2(E;m)$ satisfying 
$f\wedge 1 \in \dom$ whenever $f\in \dom$. Denote by $(\cdot,\cdot)$ and $\|\cdot\|$
the inner product and the norm in $L^2(E)$. A bilinear form $\form$ defined on 
$\dom\times \dom$ is called a lower bounded closed form on $L^2(E)$ if 
the following conditions are satisfied: there exists a $\beta\ge 0$ such that

\begin{itemize}
\item[\bf (B1)] (lower boundedness):  For any $u \in \dom$, $\form_{\beta}(u,u)\ge 0$,  
where 
\[
\form_{\beta}(u,v)=\form(u,v)+\beta (u,v), \quad u,v \in \dom.
\] 

\item[\bf (B2)] (weak sector condition): There exists a constant $K\ge 1$ so that 
\[
\big| \form(u,v)\big| \le K \sqrt{\form_{\beta}(u,u)} \cdot \sqrt{\form_{\beta}(v,v)}
\quad {\rm for} \ \ u,v \in \dom.
\]

\item[\bf (B3)] (closedness): The space $\dom$ is closed with respect to the norm 
$\sqrt{\form_{\alpha}(u,u)},\ u \in \dom$, for some $\alpha>\beta$,   or equivalently, 
for all $\alpha>\beta$.
\end{itemize}

For a lower bounded closed form $(\form,\dom)$ on $L^2(E)$ with index $\beta$, 
there exist unique strongly continuous semigroups $\{T_t; t>0\}$, $\{\hat{T}_t; t>0\}$ of linear operators on $L^2(E)$ satisfying

\begin{equation}
(T_tf,g)=(f,\hat{T}_tg) \quad \text{and} \quad
 \|T_tf\|\le e^{\beta t}\|f\|, \ \|\hat{T}_tf\|\le e^{\beta t}\|f\|
\end{equation}
for all $f, g \in L^2(E)$ and $\ t>0$,  such that their Laplace transform $G_{\alpha}$ and $\hat{G}_{\alpha}$ are determined 
for $\alpha>\beta$ by
$$
G_{\alpha} f, \hat{G}_{\alpha} f\in \dom, \ \ \form_{\alpha}(G_{\alpha}f, u)=
\form_{\alpha}(u, \hat{G}_{\alpha}f)=(f,u), \quad f\in L^2(E), \ u\in \dom.
$$
$\{T_t; t>0\}$ is said to be {\it Markovian} if $0\le T_t f \le 1, \ t>0$ whenever  
$f\in L^2(E), \ 0\le f\le 1$.  The semi-group $\{T_t; t>0\}$ is Markovian if and only if 
\begin{itemize}
\item[\bf (B4)] $\dis Uu \in \dom \quad {\rm and} \quad \form(Uu, u-Uu)\ge 0 
\quad {\rm for \ any} \ u \in \dom$, 
\end{itemize}
where $Uu$ denotes the unit contraction of $u$: $Uu=\bigl(0 \vee u\bigr)\wedge 1$.

A lower bounded closed form $(\form, \dom)$ on $L^2(E)$  with index $\beta$ satisfying {\bf (B4)}
is called a {\it lower bounded semi-Dirichlet form} on $L^2(E)$. 
The lower bounded semi-Dirichlet form $(\form, \dom)$ is said to be regular if 
$\dom \cap C_0(E)$ is uniformly dense in $C_0(E)$ and $\form_{\alpha}$-dense in $\dom$ 
for $\alpha>\beta$,  where $C_0(E)$ denotes the space of continuous functions on $E$ 
with compact support.  It is known that a Hunt process is properly associated with any regular lower bounded semi-Dirichlet form on $L^2(E)$
(see, {\it e.g.},  \cite{FOT11} and \cite{O13} for symmetric case).

\subsection{A lower bounded semi-Dirichlet form with singular drift}\label{S:semi-D}

Let $E=D$ be a bounded open set of $\rd \ (d\ge 3)$. We denote the norm on $\rd$ by $|x|$ and 
the inner product by $x \cdot y$ for $x$ and $y$.  We consider the following conditions {\bf (A1)}-{\bf (A3)} 
on the functions $A(x) = (a_{ij}(x)),  B(x)=(b_i(x))$, $C(x)=(c_i(x))$ and $k(x)$:

\begin{itemize}
\item[\bf (A1)] There exist positive constants $\alpha$ and $\beta$ such that
\[
\alpha |\xi|^2 \leq \sum_{i,j=1}^d a_{ij}(x) \xi_i \xi_j \le \beta |\xi|^2
\]
for almost every $x \in D $ and for every $\xi \in \mathbb{R}^d$. 
\item[\bf (A2)] There exists  $p_0\in (d, \infty]$ such that 
$b_i$ and $c_i \in L^{p_0}_{\sf loc}(\rd)$ for $i=1,2,\ldots, d$ and 
$k \in L^{p_0/2}_{\sf loc}(\rd)$. 
\end{itemize}
We further assume the following condition so that the corresponding semigroup is {\it Markovian} 
({\it e.g.},  \cite[(2.16)]{MR92} and \cite{O13}):

\begin{itemize}
\item[\bf (A3)]  $k(x)\ge 0$,  and $\dis - {\rm div} \, C(x) \ge 0$   in the sense that 
\[
\dis \int_{\real^d} C(x) \cdot \nabla  \varphi(x) dx\ge 0 \qquad  {\rm for \,  all }  \ \varphi \in C_0^{\infty}(\real^d)  \ {\rm with} \ \varphi \ge 0. 
\]
\end{itemize}

\medskip
Define a quadratic form as follows.  For $u,v \in C_0^{\infty}(D)$, 
\begin{equation} \label{form0} 
\form (u,v)  \dis := \int_D A \nabla u \cdot \nabla v \,  dx 
+\int_D B \cdot  \nabla u \,  v \, dx   
 +\int_D u \,  C \cdot \nabla v \, dx   +\int_D u \, v\,  k \, dx
\end{equation} 
and the Dirichlet integral by 
\[
{\mathbb D}(u,v) := \int_D \nabla u \cdot \nabla v \, dx.
\]
For $1\le p<\infty$, 
\[
W^{1,p}(D)\, :=\, \Big\{ u\in L^p(D) \ \Big|\ \frac{\partial u}{\partial x_i} \in L^p(D),\ i=1,2,\ldots,d \Big\}
\]
is the Sobolev space equipped with the norm 
\[
\| u \|_{W^{1,p}(D)}\, :=\, \bigg( \| u \|_{L^p(D)}^p 
+ \sum_{i=1}^d \Big\| \frac{\partial u}{\partial x_i} \Big\|_{L^p(D)}^p \bigg)^{\frac1p},\quad u\in W^{1,p}(D).
\]
where $\partial / \partial x_i$ are the derivatives in the sense of distributions.  We denote by $W^{1,p}_0(D)$ 
the closure of $C_0^\infty(D)$ with respect to $W^{1,p}(D)$-norm. Here $W^{1,p}(D)$ and 
$W^{1,p}_0(D)$ are Banach spaces. For $p=2$, the corresponding spaces are denoted by $H^1(D)$ and $H^1_0(D)$, 
respectively. Recall also the following Poincar\'{e} inequality holds for the bounded open set $D$:
\[
\|u\|_{L^2(D)} \leq C_P \bigl\| |\nabla u | \bigr\|_{L^2(D)} = C_P \sqrt{{\mathbb D}(u,u)},\quad u\in C_0^1(D)
\]
where $C_P$ is a positive constant. Therefore $\|u\|_{H^1(D)}^2 ={\mathbb D}(u,u)+(u,u)$ is equivalent to ${\mathbb D}(u,u)$ for $u\in H_0^1(D)$.

We will see that $(\form^\delta, H_0^1(D))$ is a regular,  lower bounded semi-Dirichlet form on 
$L^2(D) :=L^2(D;dx)$ with some (uniform) index $\beta_0>0$ for  $\delta>0$ and 
the existence and the uniqueness of solutions $u^\delta_{\lambda}$ for each $f\in H^{-1}(D)$ and 
$\lambda >\beta_0$ to the following equation:
\begin{equation} \label{main eq.}
\form^\delta (u^\delta_{\lambda}, v) +\lambda (u^\delta_{\lambda}, v)_{L^2(D)}\, =\, \langle f,v \rangle_{H^{-1}, H^1_0}, 
\qquad v\in H_0^1(D).
\end{equation}

Before we construct a semi-Dirichlet form, we prepare the following lemma.

\begin{lemma}  \label{lem:est1}
Assume {\bf (A2)} holds for the function $B(x)$ and $k(x)$. 
Then, for any $\lambda>0$, there exists a constant $\beta_0>0 \, ($depending on $d$, $\lambda$,  $\| \, B \, \|_{L^{p_0}(D)}$ 
and $\| k \|_{L^{p_0/2}(D)})$  so that 
\begin{equation} \label{est-drift}
\Big| \int_D B(x) \cdot \nabla u(x) u(x) dx\Big| \le  \lambda {\mathbb D}(u,u) + \beta_0 \int_D u(x)^2 dx
\end{equation}
and 
\begin{equation} \label{est-pot}
\Big| \int_D u(x)^2 k(x) dx \Big| \le \lambda {\mathbb D}(u,u) + \beta_0 \int_D u(x)^2 dx, \quad 
u \in C_0^{\infty}(D).
\end{equation}
\end{lemma}

\noindent
\proof. We first show \eqref{est-drift}  in the case where $p_0=\infty$.  
For $u \in C_0^{\infty}(D)$ and $\vareps>0$, we find that 
\begin{equation} \label{est1} 
\begin{split}
\Big| \int_D B(x) \cdot \nabla u(x) u(x)dx  \Big|  
& \le \sum_{i=1}^d \int_D  \big| b_i(x) \partial_i u(x) u(x)\big| dx \\ 
&\le \sum_{i=1}^d \|b_i\|_{L^\infty(D)}  \int_D \big|\partial_i u(x) u(x)\big| dx \\ 
& \le \frac{\dis \max_{1\le i\le d}\|b_i\|_{L^\infty(D)}}2    \sum_{i=1}^d \Big( \vareps 
\int_D \big|\partial_iu(x)\big|^2 dx + \frac 1{\vareps} \int_D u(x)^2 dx \Big) \\
&= \frac{\dis \max_{1\le i\le d}\|b_i\|_{L^\infty(D)}}2  \Big( \vareps {\mathbb D}(u,u) +  \frac d{\vareps} 
\int_D u(x)^2 dx \Big).
\end{split}
\end{equation}
So, taking $\dis \vareps=\lambda/\Big\{ \frac{\max_{1\le i\le d} \|b_i\|_{L^\infty(D)}}2+1\Big\}>0$ 
in \eqref{est1}  for any $\lambda>0$,
$$
\Big| \int_D B(x) \cdot \nabla u(x) u(x) dx\Big| \le  \lambda {\mathbb D}(u,u) + \beta_0 \int_D u(x)^2 dx
$$
holds if we set, for example, $\beta_0=d/\lambda$. Next, we consider the case $d<p_0<\infty$. As in the previous case, we see that for any $\vareps>0$,
\begin{equation} \label{est3} 
\begin{split}
\Big| \int_D B(x)\cdot \nabla u(x) u(x)dx  \Big|  
& \le \sum_{i=1}^d \int_D  \big| b_i(x) \partial_i u(x) u(x)\big| dx   \\
& \le \frac 12 \sum_{i=1}^d \Big( \vareps \int_D |\partial_iu(x)|^2 dx + \frac 1{\vareps} \int_D  b_i(x)^2 u(x)^2 dx \Big) \\
& =   \frac {\vareps}2 {\mathbb D}(u,u) + \frac 1{2\vareps} \sum_{i=1}^d \int_D b_i(x)^2 u(x)^2dx.
\end{split}
\end{equation}
We estimate the second term of the right-hand side of \eqref{est3}.  To this end, we prepare auxiliary positive constants 
$\gamma$ and $p$ with $0<\gamma<2$ and $1<p<p_0/2$ in addition to $\vareps>0$ as follows:
\begin{equation}\label{est4}
\begin{split}
 \sum_{i=1}^d \int_D b_i(x)^2 u(x)^2dx   & = \sum_{i=1}^d \int_D  (\vareps b_i(x)^2 |u(x)|^{\gamma}) \cdot \frac{|u(x)|^{2-\gamma}}{\vareps}  dx   \\ 
& \le \sum_{i=1}^d  \int_D\Big(\frac{ \vareps^p b_i(x)^{2p} |u(x)|^{\gamma p}}p + \frac{|u(x)|^{(2-\gamma)q}}{q\vareps^q} \Big)dx \\
&=\frac{\vareps^p}{p} \sum_{i=1}^d \int_D b_i(x)^{2p} |u(x)|^{\gamma p}dx + \frac d{q \vareps^q} \int_D |u(x)|^{(2-\gamma) q}dx.
\end{split}
\end{equation}
The above inequality follows from the Young inequality for $p$ and $q$ with $1/p+1/q=1$ ($xy \le x^p/p+y^q/q$ for $x,y\ge 0$).  
Then taking $p$ and $\gamma$ to satisfy $(2-\gamma)q=2$, we see that $p$ is equal to $2/\gamma$.  
So, the first term of the right-hand side is estimated by 
\[ 
\begin{split}
& \frac{\vareps^p}{p} \sum_{i=1}^d \int_D b_i(x)^{2p} |u(x)|^{\gamma p}dx \\
& \qquad\; = \frac{\gamma \vareps^{2/\gamma}}2 \sum_{i=1}^d \int_D b_i(x)^{4/\gamma} u(x)^2dx \\
& \qquad\; \le \frac{\gamma \vareps^{2/\gamma}}2 \sum_{i=1}^d  \Big(\int_D |b_i(x)|^{2d/\gamma}dx\Big)^{2/d} 
\Big(\int_D |u(x)|^{2d/(d-2)}dx\Big)^{(d-2)/d},
\end{split}
\]
where we used the H$\ddot{\rm o}$lder inequality in the last inequality with $s=d/2$ and $t=d/(d-2)$  with $1/s+1/t=1$.
Thus if we  set $\gamma :=2d/p_0$ (and then $p=p_0/d$), $0<\gamma<2$ and $1<p<p_0/2$ hold. 
By the estimates above,  we have 
\begin{equation} \label{est5} 
\begin{split}
&  \Big|  \int_D B(x) \cdot \nabla u(x) u(x) dx \Big|    \\
\quad &  \le \frac {\vareps}2 {\mathbb D}(u,u)  
 + \frac 1{2\vareps}  \Big(
\frac{\vareps^{p_0/d}}{p_0} \sum_{i=1}^d \|b_i\|_{L^{p_0}(D)}^{2p_0/d}  \|u\|_{L^{2d/(d-2)}(D)}^2  
 + \frac{d(p_0-d)}{p_0 \vareps^{p_0/(p_0-d)}} \|u\|_{L^2(D)}^2 \Big)
\end{split}
\end{equation}
for any $\vareps>0$. As for the $L^{2d/(d-2)}$-norm of $u$,  
recall the Gagliardo-Nirenberg-Sobolev inequality (\cite[Theorem 4.5.1]{EG92}):  
\begin{equation} \label{eq:GNS}
\Big(\int_D |u(x)|^{pd/(d-p)} dx \Big)^{(d-p)/(pd)} \le K_1 \Big(\int_D |\nabla u(x)|^p dx\Big)^{1/p}
\end{equation}
for $u \in W^{1,p}(D)$ with $1\le p<d$ and $K_1$ depends only on $p$ and $d$.
Then taking $p=2$ in \eqref{eq:GNS},  the right hand side of \eqref{est5} is estimated as 
\[  \label{est6} 
\begin{split}
 \Big|  \int_D B(x) \cdot \nabla u(x) u(x) dx \Big|   & \le  \frac{\vareps}{2} {\mathbb D}(u,u) + \frac 1{2\vareps}\frac{\vareps^{p_0/d}}{p_0} \sum_{i=1}^d \|b_i\|_{L^{p_0}(D)}^{2p_0/d}  K_1^2 {\mathbb D}(u,u)     \\
& \quad + \frac 1{2\vareps}  \frac{d(p_0-d)}{p_0 \vareps^{p_0/(p_0-d)}} \|u\|_{L^2(D)}^2 \\
 & =\Big(\frac{\vareps}2 +\frac{\vareps^{p_0/d-1}}{2p_0} K_1^2  \sum_{i=1}^d \|b_i\|_{L^{p_0}(D)}^{2p_0/d}\Big) 
 {\mathbb D}(u,u)  \\ 
 & \quad + \frac{d(p_0-d)}{2p_0 \vareps^{1+p_0/(p_0-d)}} \|u\|_{L^2(D)}^2.
 \end{split}
\]
Finally, for any $\lambda>0$, noting $p_0>d$, we can take $\vareps>0$ so that 
\[
\frac{\vareps}2 +\frac{\vareps^{p_0/d-1}}{2p_0} K_1^2  \sum_{i=1}^d \|b_i\|_{L^{p_0}(D)}^{2p_0/d}<\lambda
\]
and then putting  $\beta_0=\frac{d(p_0-d)}{2p_0 \vareps^{1+p_0/(p_0-d)}}$, 
the desired inequality \eqref{est-drift} holds. As for the inequality \eqref{est-pot}, it follows from the subsequent argument of 
\eqref{est4}. \hfill \fbox{}

\smallskip

\begin{remark} Note that the following inequality also holds:
$$
 \Big|\int_D B(x) \cdot \nabla u(x) u(x)dx \Big| 
 \le \frac{2(d-1)}{d-2} \Big(\sum_{i=1}^d \|b_i\|^2_{L^d(D)}\Big)^{1/2} {\mathbb D}(u,u), \quad u \in C_0^{\infty}(D).
$$
\end{remark}

\smallskip
Now, we construct a lower bounded closed form on $L^2(D)$ with singular drifts and potential. 
\begin{prop} \label{DF} 
Assume {\bf (A1)} holds for $A(x)$, {\bf (A2)}  holds for $B(x), C(x)$ and $k(x)$.  
The quadratic form  $\form$ defined in \eqref{form0} satisfies the conditions 
{\bf (B1)} and {\bf (B2)}  for $\dom=C_0^{\infty}(D)$ and some constant $\beta_0\ge 0$. 
\end{prop} 

\proof  From the condition {\bf (A1)},  we see that
$$
\alpha {\mathbb D}(u,u)\, \leq\, \int_D A(x) \nabla u(x)\cdot \nabla u(x) dx
\le \beta{\mathbb D}(u,u), \quad  u \in C_0^{\infty}(D).
$$
According to the lemma \ref{lem:est1}, for each $\lambda>0$ there exists $\beta_0=\beta_0(\lambda)>0$ so that 
\[
\begin{split} 
& \max\Big\{ \Big|\int\limits_D B(x)  \cdot \nabla u(x) u(x)dx \Big|,  
\Big|\int\limits_D u(x) \, C(x)\cdot \nabla  u(x)dx \Big|, 
\Big| \int\limits_D u(x)^2 k(x)dx \Big| \Big\}  \\
& \qquad  \le \lambda {\mathbb D}(u,u) + \beta_0 \|u\|^2_{L^2(D)}, 
  \end{split}
\]
whence {\bf (B1)} holds so that $\form(u,u)+\beta_0(u,u) \ge 0, \ u \in C_0^{\infty}(D)$. 
Note also that 
$$
(\gamma \wedge 1) \|u\|_{H^1(D)}^2 \le  {\mathbb D}(u,u)+ \gamma \|u\|_{L^2(D)}^2 \le 
(\gamma \vee 1) \|u\|_{H^1(D)}^2, \quad  u\in H^1(D)
$$
holds for any $\gamma>0$.  Then we can find constants $\kappa_1$ and $\kappa_2$ for each $\beta>\beta_0$ 
so that 
\begin{equation}  \label{comparison} 
\kappa_1 \| u\|_{H^1(D)}^2 \le \form(u,u) +\beta (u,u) \le \kappa_2 \| u\|_{H^1(D)}^2, \ u\in C_0^{\infty}(D).
\end{equation}
Using this inequality with Lemma \ref{lem:est1},  we see that the condition {\bf (B2)} holds.
 \hfill \fbox{}

\bigskip
The estimate \eqref{comparison} shows that $\sqrt{\form_{\beta}(\cdot, \cdot)}$ for $\beta>\beta_0$ and 
the Sobolev norm $\|\cdot\|_{H^1(D)}$ are equivalent. So, taking the closure of $C_0^{\infty}(D)$ with respect to 
$\| \cdot \|_{H^1(D)}^2$-norm, \eqref{comparison} also holds for $u\in H_0^1(D)$.  Then, we can conclude that 
{\bf (B3)} holds with $\dom=H_0^1(D)$.  
Moreover, by the {\rm Poincar\'{e}} inequality, we see that $\form_\beta(\cdot, \cdot)$ 
and the Dirichlet integral ${\mathbb D}$ are also equivalent.  Then, we can obtain the following theorem.

\begin{theorem} \label{prop:form0}
Assume {\bf (A1)-(A2)} hold. Then the form $\form$ defined by \eqref{form0} extends from 
$C_0^{\infty}(D)\times C_0^{\infty}(D)$ to $H^1_0(D) \times H^1_0(D)$ to be a regular lower bounded closed 
form on $L^2(D)$ satisfying {\bf (B1)-(B3)} with $\dom=H_0^1(D)$ and 
the parameter $\beta_0$ appeared in the proof of Proposition {\rm \ref{DF}}.
\end{theorem}

\medskip
We further assume {\bf (A3)} holds  for $C(x)$ and $k(x)$ in addition to {\bf (A1)} and  {\bf (A2)}.  
Recall that $Uu$ is the unit contraction of a function $u$ on $D$, 
$Uu(x):=\big( 0 \vee  u(x) \big) \wedge 1, \ x\in D$. 
Then, it is known that $Un \in H^1_0(D)$ if $u\in H^1_0(D)$ and, we find that,  by \cite[\S II,2]{MR92} 
or \cite[\S 1.5]{O13},  $\form(Uu, u-Uu)\ge 0$ holds.
Therefore, the previous proposition readily leads us to the following result.
\begin{cor}
Assume {\bf (A1)-(A3)} hold. Then the form $(\form, H^1_0(D))$ is a regular lower bounded 
semi-Dirichlet form on $L^2(D)$ with the index $\beta_0$ appeared in the proof of Proposition {\rm \ref{DF}}.
Namely, the pair $(\form, H^1_0(D))$  satisfies  {\bf (B1)-(B3)} and {\bf (B4)} as well. 
\end{cor}

\subsection[tdsfasfsa]{A sequence of lower bounded closed forms  having singular drifts and potentials}

In this subsection we consider  that the functions $A^\delta(x)=(a_{ij}^\delta(x))$, $B^\delta(x)=(b_i^\delta(x))$,
 $C^\delta(x)=(c_i^\delta(x))$  and $k^\delta(x)$ defined on $\rd$ satisfy the following conditions 
 {\bf (A1)$'$-(A3)$'$} instead of {\bf (A1)-(A3)} for $\delta>0$:

\begin{itemize}
\item[\bf (A1)$'$]
 There exist positive constants $\alpha$ and $\beta$, 
$$
\alpha |\xi|^2 \leq \sum_{i,j=1}^d a_{ij}^\delta(x) \xi_i \xi_j \le \beta |\xi|^2
$$
for almost every $x \in D $ and every $\xi \in \mathbb{R}^d$.  

\item[\bf (A2)$'$] There exists  $p_0\in (d, \infty]$ such that $b_i^\delta$ and $c_i^\delta$ belong to 
$L^{p_0}_{\sf loc}(\rd)$ for $i=1,2,\ldots, d$, and $k^\delta$ belongs
  to $L^{p_0/2}_{\sf loc}(\rd)$ such that 
\[ 
\sup_{1\le i \le d \atop \delta>0} \|b_{i}^\delta\|_{L^{p_0}(F)}<\infty, \ 
\sup_{1\le i \le d \atop \delta>0} \|c_{i}^\delta\|_{L^{p_0}(F)}<\infty \ {\rm and} \ 
\sup_{\delta>0} \| k^\delta\|_{L^{p_0/2}(F)}<\infty
\]
 for each  compact  set  $F\subset \real^d$.
\end{itemize}
We also consider the condition.
\begin{itemize}
\item[\bf (A3)$'$]   $k^\delta(x)\ge 0$,  and $\dis - {\rm div} \, C^\delta(x) \ge 0$ in the sense that
\[ 
\dis \int_{\real^d} C^\delta(x) \cdot \nabla  \varphi(x) dx\ge 0 \qquad  {\rm for \, all}  \ 
\varphi \in C_0^{\infty}(\real^d) \,\, \text{with} \,\, \varphi \ge 0. 
\] 
\end{itemize}
For $\delta>0$, define quadratic forms on $L^2(D)$ as follows: for $u,v \in H^1_0(D)$, 
\begin{align} \label{form} \nonumber 
\form^\delta (u,v) & := \int_D A^\delta (x) \nabla u(x) \cdot \nabla v(x) dx 
+\int_D B^\delta (x) \cdot \nabla u(x) v(x) dx  \\
& \quad \;  +\int_D u(x) C^\delta (x) \cdot \nabla v(x) dx  + \int_D u(x) v(x) k^\delta(x)dx.
\end{align} 
Replacing $\form$ to $\form^{\delta}$ (that is, replacing $(A, B, C,k)$ to $(A^{\delta}, B^{\delta}, 
C^{\delta}, k^\delta)$ in the expression of \eqref{form0}),  Theorem \ref{prop:form0} and its corollary are modified as follows.

\begin{theorem} \label{delta-DF} 
Assume {\bf (A1)$'$-(A2)$'$} hold.  For each $\delta >0$, $(\form^\delta, H_0^1(D))$ is a  regular lower bounded 
closed form on $L^2(D)$ with (universal) index $\beta_0$.
\end{theorem} 

\begin{cor}
Assume {\bf (A1)$'$-(A3)$'$} hold.  Then, for each $\delta>0$,   the form $(\form^\delta, H^1_0(D))$ is a regular lower bounded 
semi-Dirichlet form on $L^2(D)$ with the index $\beta_0$.
\end{cor}

\begin{remark} Suppose that $B(x)$ satisfy {\bf (A2)} and set $B^\delta(x)=B(x/\delta)$ for $\delta>0$.
Then {\bf (A2)$'$} hold for $\{B^\delta\}_{\delta>0}$ provided that $x\mapsto B(x)$ is $Y$-periodic. We impose the 
condition {\bf (A2)$'$} since $x\mapsto B(x)$  is not necessarily $Y$-periodic.
\end{remark}

We now show the existence of solutions and derive a priori estimate of the solutions.

\begin{theorem} \label{thm:exist} 
Assume {\bf (A1)$'$-(A2)$'$} hold.   
Then for any $\delta>0$, \ $\lambda>\beta_0$  and $f\in H^{-1}(D)$, there exists a unique 
element $u^{\delta}_{\lambda} \in H^1_0(D)$ such that 
\begin{equation} \label{apriori1}
\form^{\delta}_{\lambda}(u^{\delta}_{\lambda}, \varphi)=\langle f, \varphi \rangle_{H^{-1}, H^1_0}, 
\quad  \varphi \in H_0^1(D)
\end{equation}
and,  the sequence $\{u^{\delta}_{\lambda}, \delta>0\}$  satisfies that 
\begin{equation} \label{apriori2}
\\|u^{\delta}_{\lambda}\|_{H^1} \le c_1 \|f\|_{H^{-1}}
\end{equation}
for some constant $c_1>0$. 
\end{theorem}

\begin{proof}
The existence and the uniqueness of solutions are easy. In fact,  since $\form^{\delta}_{\lambda}$ is a bounded bilinear 
functional on $H^1_0(D)$ and satisfies  the coerciveness,  
there exists a unique $u^{\delta}_{\lambda} \in H^1_0(D)$ for each $f\in H^{-1}(D)$,  $\delta>0$ and $\lambda>\beta_0$ 
such that \eqref{apriori1} holds by the Lax-Milgram theorem.
As for the priori estimate \eqref{apriori2},  plugging $\varphi=u^{\delta}_{\lambda}$ in \eqref{apriori1} 
as a test function, we find that 
$$
\form^{\delta}_{\lambda}(u^{\delta}_{\lambda}, u^{\delta}_{\lambda}) 
= \big|\langle f, u^{\delta}_{\lambda} \rangle_{H^{-1}, H^1_0} \big| \le \|f\|_{H^{-1}} \cdot \|u^{\delta}_{\lambda}\|_{H^1}
$$
holds ({\it c.f.} \cite[\S 2]{TU22}).  Here $\|f\|_{H^{-1}}$ is the continuity constant of $f$ as a functional of $H^1_0(D)$.
As for the left-hand side,  we see that, from \eqref{comparison}, 
$$
\kappa_1 \|u^{\delta}_{\lambda} \|_{H^1}^2  \le \form^{\delta}_{\lambda} (u^{\delta}_{\lambda}, u^{\delta}_{\lambda}).
$$
So
$$
\kappa_1  \|u^{\delta}_{\lambda}\|_{H^1}^2 \le  \|f\|_{H^{-1}} \cdot \|u^{\delta}_{\lambda}\|_{H^1}
$$
and,  putting $c_1=1/\kappa_1$, we see that \eqref{apriori2} holds. 
\end{proof}

\bigskip
Since $\{u^{\delta}_{\lambda}\}_{\delta>0}$ is a bounded sequence in $H^1_0(D)$,  we find that, 
up to a subsequence,  $\{u^{\delta}_{\lambda}\}$ converges to a limit $u_0$ {\it weakly} in $H_0^1(D)$ and 
it also converges to $u_0$ {\it strongly} in $L^2(D)$ by the fact that $H_0^1(D)$ is compactly embedded into $L^2(D)$.
Then, our next goal is to show the limit's uniqueness and find the homogenized equation that is satisfied by it.
To this end, we prepare some notions and introduce the unfolding operators. 
\section{Unfolding Operators}\label{sec:unfold}

\subsection{Periodic Sobolev Spaces}

Recall that $Y=(0,1)^d$ is the open unit cube in $\rd$.  Define a periodic 
Sobolev space $W^{1,p}_{\#}(Y)$ for $1\le p<\infty$ as the closure of $C_{\#}^{\infty}(Y)$, the restrictions of 
smooth functions defined on $\rd$ which are $Y$-periodic,  with respect to the 
$W^{1,p}(Y)$-norm, and then it becomes a Banach space with the same $W^{1,p}(Y)$-norm. 
For $p=2$, corresponding space $W^{1,2}_{\#}(Y)$ is also denoted by $H_{\#}(Y)$:
$$
H_{\#}(Y) \ = \ \mbox{the  closure of $C_{\#}^{\infty}(Y)$ \ with respect to  $H^1(Y)$-norm}.
$$
Define also $C_{\#}(Y)$ the space of all continuous functions defined on 
 $\rd$ which are $Y$-periodic.  Then, the following results are known. 

\begin{lemma} {\rm (\cite[Proposition\ 3.49]{CD10})} 
If $u\in H_{\#}(Y)$, then the trace of $u$ on the ``opposite faces of $Y$" are equal.
\end{lemma}

We define the quotient space ${\mathcal W}^{1,p}_{\#}(Y):= {W}^{1,p}_{\#}(Y)/\real$ of 
$W^{1,p}_{\#}(Y)$ relative to the following relation ``$\sim$":  \ for 
$u,v \in W^{1,p}_{\#}(Y)$, 
$$
u \sim v \quad \Longleftrightarrow \quad u-v = \ {\rm constant} \ {\rm a.e. \ on}  \ Y.
$$
In other words, the space ${\mathcal W}^{1,p}_{\#}(Y)$ is the collection of all equivalent classes 
$\tilde{u}$ of functions $u$ in $W^{1,p}_{\#}(Y)$:
$$
\tilde{u}:=\big\{ v\in W^{1,p}_{\#}(Y) : u\sim v \big\}, \quad u \in W^{1,p}_{\#}(Y).
$$
For the case $p=2$, we also use ${\mathcal H}_{\#}(Y)$ instead of ${\mathcal W}^{1,2}_{\#}(Y)$.

\begin{lemma} {\rm (\cite[Proposition\ 3.52]{CD10})}  The space ${\mathcal H}_{\#}(Y)$ is a 
Hilbert space under the following inner product:
$$
\big(\! \big( \tilde{u}, \tilde{v} \big)\! \big) := \int_Y \nabla f(x) \cdot \nabla g(x)dx, \quad 
f \in \tilde{u},\ g \in  \tilde{v}, \ \tilde{u}, \tilde{v} \in {\mathcal H}_{\#}(Y).
$$
Moreover the dual space $\big({\mathcal H}_{\#}(Y)\big)'$ of ${\mathcal H}_{\#}(Y)$ is identified 
with the set
$$
\Big\{ \ell \in \big(H_{\#}(Y)\big)' : \ \ell(c)=0, \ c\in \real \Big\},
$$
where 
$$
\langle \ell, \tilde{u}\rangle_{({\mathcal H}_{\#}(Y))',\, {\mathcal H}_{\#}(Y)} =
\langle \ell, f \rangle_{(H_{\#}(Y))',\, H_{\#}(Y)} \quad  {\rm for} \quad  f \in \tilde{u}.
$$
\end{lemma}

We now show the existence of functions, so-called the {\it correctors}, that appear as solutions of variational formulas, which play an important role in identifying the limits of the solutions to homogenization problems.
 
\begin{prop}[Correctors] \label{prop:corr}
Let $A(x,y)=(a_{ij}(x,y))$ and $C(x,y)=(c_i(x,y))$ be functions defined on 
$D\times Y$ satisfying the following conditions: 
\begin{equation} \label{corre1}
\alpha |\xi|^2 \le \sum_{i,j=1}^d a_{ij}(x,y) \xi_i \xi_j \le \beta |\xi|^2, \quad \forall \, \xi \in \real^d, \ {\rm a.e.} \ (x,y) \in 
D\times Y
\end{equation}
and there exists $p_0 \in (d, \infty]$ such that 
\begin{equation} \label{corre2}
c_i \in L^{p_0}(D\times Y),  \quad i=1,2,\ldots, d.
\end{equation}
Then the following hold for a.e.  $x\in D$:
\begin{itemize}
\item[\bf (1)]  For each $i=1,2,\ldots,d$, there exists a unique (up to an additive constant) $\omega_i \in L^2(D;H_{\#}(Y))$ so that
\begin{equation} \label{effect}
\int_Y A(x,y) \nabla_y \omega_i(x, y) \cdot \nabla \varphi(y)dy = -\int_Y A(x, y) {\bm e}_i \cdot \nabla \varphi(y)dy, \quad 
\varphi \in H_{\#}(Y).
\end{equation}
\item[\bf (2)] 
There exists a unique (up to an additive constant) $\omega_0 \in L^2(D;H_{\#}(Y))$ so that 
\begin{equation} \label{effect1}
\int_Y A(x, y) \nabla_y \omega_0(x, y) \cdot \nabla \varphi(y)dy = -\int_Y C(x, y) \cdot \nabla \varphi(y)dy, \quad 
\varphi \in H_{\#}(Y).
\end{equation}
\end{itemize}
Moreover, such $\omega_i$ is unique up to an additive constant with respect to $y$ and 
$\big\| |\nabla_y \,\omega_i(x, \cdot)|\big\|_{L^2(Y)}$ belong to $L^\infty(D)$ for $1\le i \le d$. 
If $C \in (L^\infty(D; L^{p_0}(Y)))^d$ holds, then $\big\| |\nabla_y \,\omega_0(x, \cdot)|\big\|_{L^2(Y)}$ 
also  belongs to $L^\infty(D)$.
\end{prop}

\begin{proof}
 These are shown by using the Lax-Milgram theorem in $H:=L^2(D; H_{\#}(Y))$. 
 Set
$$
\ell_i(v):=  -\iint_{D\times Y} A(x, y){\bm e}_i \cdot \nabla_y v(x, y) dydx, \quad v\in H \quad i=1,2,\ldots,d, 
$$
and 
$$
\ell_0(v):= -\iint_{D\times Y} C(x, y) \cdot \nabla_y v(x, y) dydx, \quad v\in H.
$$
Then we find that $\ell_i(v)=\ell_i(v')$ if $v$ and $v'$ are in $H$ so that $v-v'$ is constant with respect to $y$ 
for a.e. $x\in D$.  In other words, $\ell_i$ is a linear functional on $\tilde{H}:=L^2(D; {\mathcal H}_{\#}(Y))$, 
not just on $H$.  According to \eqref{corre1}  for $i=1,2,\ldots, d$, and Lemma \ref{lem:est1} and \eqref{corre2} 
for $i=0$,  we see
\[
\big|\ell_i(v)\big|  \le c \sqrt{\iint_{D\times Y} \nabla_y v(x, y) \cdot \nabla_y v(x, y) dy dx}  
 = c \|v\|_H,  \quad v\in H \ \  ({\rm or} \ \tilde{v} \in \tilde{H})
\]
for some constant $c>0$.  That is, $\ell_i$ are continuous on $H$. 
On the other hand, we also see that 
$$
\tilde{\form}(u,v):=\iint_{D\times Y} A(x, y) \nabla_y u(x, y)\cdot \nabla_y v(x, y) dydx, \quad u,v  \in \tilde{H}
$$
defines a coercive bilinear form on $\tilde{H}$, which is also continuous. Thus, by the Lax-Milgram theorem,  
there exists an $\omega_i \in H$ for each $i=0,1,2,\ldots, d$ so that 
$$
\tilde{\form}(\omega_i, v)=\ell_i(v), \quad v \in H.
$$
Thus, we have for all $\varphi \in H$, 
\begin{align*}
\tilde{\form}(\omega_i, \varphi) 
& =\iint\limits_{D\times Y} A(x, y) \nabla_y  \omega_i(x, y) \cdot \nabla_y \varphi(x,y) dydx  \\
& =- \iint\limits_{D \times Y} A(x, y){\bm e}_i \cdot \nabla_y \varphi(x, y) dydx  
\ =\ell_i(\varphi),  \qquad i=1,2,\ldots, d
\end{align*}
and 
\[
\tilde{\form}(\omega_0, \varphi)  =\iint\limits_{D\times Y} A(x, y) \nabla_y  \omega_0(x, y) 
\cdot \nabla_y \varphi(x, y) dydx 
 =-\iint\limits_{D\times Y} C(x, y) \cdot \nabla_y \varphi (x, y) dydx 
=\ell_0(\varphi).
\]
On the other hand,  plugging $\varphi(x,y)=\varphi_1(x) \varphi_2(y)$ for $\varphi_1 \in C_0^\infty(D)$ and $\varphi_2 \in H_{\#}(Y)$ into 
the above equations as test functions,  it follows that \eqref{effect} and \eqref{effect1} hold by the denseness of $C_0^\infty(D)\times H_{\#}(Y)$ in 
$H=L^2(D;H_{\#}(Y))$. 

Finally, because of the definition of ${\mathcal H}_{\#}(Y)$,  we can conclude that $\omega_i$ is unique up to an additive 
constant (with respect to $y$) for $i=0,1,2,\ldots, d$.  As for the boundedness of the derivative of $\omega_i(x,y)$ 
with respect to $y$, we only show for $w_0$ because of the proofs for the other $\omega_i$ are much simpler.  
Since $\omega_0(x,y)$ belongs to $H_{\#}(Y)$ as a function of $y$ for a.e. $x\in D$, 
plugging $\varphi(y)=\omega_0(x, y)$ in \eqref{effect1} as a test function, we  see that 
\begin{align*}
& \alpha \big\| |\nabla_y \omega_0(x, \cdot)|\big\|_{L^2(Y)}^2  \\
& \qquad =  \alpha \int_Y |\nabla_y \omega_0(x, y)|^2dy  \ \le \int_Y A(x,y) \nabla_y \omega_0(x, y) \cdot \nabla_y \omega_0(x, y)dy \\
&\qquad  = -\int_Y C(x,y) \nabla_y \, \omega_0(x,y)dy \ = - \sum_{i=1}^d \int_Y c_i(x,y) \frac{\partial \omega_0}{\partial y_i}(x,y)dy  \\
&\qquad \le \big\| |\nabla_y \omega(x, \cdot)|\big\|_{L^2(Y)}
 \sqrt{\sum_{i=1}^d  \int_Y c_i(x,y)^2 dy}   \\
&\qquad \le \big\| |\nabla_y \omega(x, \cdot)|\big\|_{L^2(Y)} \sqrt{ \sum_{i=1}^d \Big(\int_Y |c_i(x,y)|^{p_0}dy \Big)^{2/p_0} 
\Big(\int_Y 1^{p_0/(p_0-2)}dy \Big)^{(p_0-2)/p_0} }.
\end{align*}
So the inequality 
$$
\big\| |\nabla_y \omega_0(x, \cdot)|\big\|_{L^2(Y)} \le \frac 1\alpha \Big( \sum_{i=1}^d \Big(\int_Y |c_i(x,y)|^{p_0}dy \Big)^{2/p_0}\Big)^{1/2}
$$
holds and the right hand side is bounded in $x$ by the assumption $c_i \in L^\infty(D; L^{p_0}(Y))$.
\end{proof}

\subsection{Unfolding Operators}

For each $x\in \rd$,  one has $x= [x]+ \{x\}$, where $[x]=(m_1,m_2,\ldots,m_d)^{\sf T} \in{\mathbb Z}^d$ denotes the  
unique integer combination $\sum_{j=1}^d m_j {\bm e}_j$ such that $\{x\}=x-[x]$ belongs to 
$Y$ a.e.   Then for a.e.  $x\in \rd$ and $\delta>0$, 
$$
x=\delta \Big( \Big[ \frac x{\delta} \Big] +\Big\{ \frac{x}{\delta}\Big\} \Big).
$$
Following \cite{CDG18}, define 
\[
 \Xi_{\delta} ;= \Big\{ \xi \in{\mathbb Z}^d : \  \delta(\xi +Y) \subset D\Big\} \quad
\widehat{D}_{\delta}:= \text{interior} \Big\{ \bigcup_{\xi \in \Xi_{\delta}} \delta \big(\xi+\overline{Y}\big)
\Big\}, \quad \text{and}  \quad \Lambda_{\delta} :=D\setminus \widehat{D}_{\delta}.
\] 
The set $\widehat{D}_{\delta}$ is the largest union of $\delta \big(\xi +\overline{Y}\big)$ 
cells for  $\xi \in {\mathbb Z}^d$ included in $D$.

\begin{defi} For a measurable function $\phi$ on $\widehat{D}_{\delta}$,  the unfolding operator 
${\mathfrak T}_{\delta}$ is defined as 
\begin{equation} \label{unfold}
{\mathfrak T}_{\delta}(\phi)(x,y) =
\left\{
\begin{array}{cc}
\dis \phi\Big(\delta \Big[\frac{x}{\delta}\Big] +\delta y \Big),  &   
{\rm for} \ {\rm a.e.} \ \ (x,y) \in \widehat{D}_{\delta}\times Y,  \\
 \vspace*{-5pt} \\
0,  & {\rm for} \ {\rm a.e.}  \ \ (x,y) \in \Lambda_{\delta} \times Y.
\end{array}
\right.
\end{equation}
\end{defi}

Note that the unfolding operator ${\mathfrak T}_{\delta}$ satisfies the following properties 
for measurable functions $u$ and $w$:
$$
{\mathfrak T}_{\delta}(u w)={\mathfrak T}_{\delta}(u)  {\mathfrak T}_{\delta}(w), \qquad 
{\mathfrak T}_{\delta}\big(f(u)) =f({\mathfrak T}_{\delta}(v)),
$$
where $f$ is a continuous function on ${\mathbb R}$ with $f(0)=0$.
\begin{prop}  {\rm (\cite[Proposition 1.5]{CDG18})}
Let $f$ be a measurable function on $Y$ and extend it by $Y$-periodicity to the whole space $\rd$ as 
$f(x) =f(\{x\}), \  {\rm a.e.}  \  x\in \rd$. Then, defining the sequence $\{f_{\delta}\}$ by 
$$
f_{\delta}(x) =f\Big(\frac {x}{\delta}\Big), \quad {\rm a.e.} \ \ x\in \rd, 
$$
we have
$$
{\mathfrak T}_{\delta}(f_{\delta}|_D)(x,y)= \left\{
\begin{array}{ll}
f(y), & {\rm for \ \ a.e.} \ \ (x,y) \in \widehat{D}_{\delta}\times Y, \\
 \vspace*{-5pt} \\
0, & {\rm for \ \ a.e.} \ \ (x,y) \in \Lambda_{\delta}\times Y.
\end{array}
\right.
$$
Moreover, if $f\in L^p(Y) \ (1\le p<\infty)$,   ${\mathfrak T}_{\delta}(f_{\delta}|_D)$ converges 
to $f$ strongly in $L^p(D\times Y)$.
\end{prop}

\begin{prop} {\rm (\cite[Proposition1.8]{CDG18})} Suppose $1\le p\le \infty.$ The operator 
${\mathfrak T}_{\delta}$ is a linear and continuous from $L^p(\widehat{D}_{\delta})$ into 
$L^p(D\times Y)$. For  $u \in L^1(\widehat{D}_{\delta})$,  $v \in L^1(D)$ and $w\in L^p(D)$, 
\begin{itemize}
\item[(i)] \   $\dis \iint_{D\times Y} {\mathfrak T}_{\delta}(u)(x,y)dxdy 
=\int_{\widehat{D}_{\delta}} u(x)dx;$
\item[(ii)] \  $\dis \Big| \int_D v(x)dx - \iint_{D\times Y} 
{\mathfrak T}_{\delta}(v)(x,y) dxdy \Big| \le  \int_{\Lambda_{\delta}} |v(x)|dx;$
\item[(iii)] \   $\dis \| {\mathfrak T}_{\delta}(w)\|_{L^p(D\times Y)} \le \|w\|_{L^p(D)}$.
\end{itemize}
\end{prop}

\begin{prop} {\rm (\cite[Proposition1.9]{CDG18})} \label{propconverge}
Suppose $1\le p<\infty.$ 
\begin{itemize}
\item[(i)] \  For $u\in L^p(D)$,   ${\mathfrak T}_{\delta}(u)$  converges  to $u$  
{\it strongly  in}  $ L^p(D\times Y);$
\item[(ii)] \  Let $\{w_{\delta}\}$ be a sequence in $L^p(D)$ such that 
$w_{\delta}$ converges to $w$ strongly in $L^p(D)$. Then 
${\mathfrak T}_{\delta}(w_{\delta})$  converges  to $w$  
{\it strongly  in}  $ L^p(D\times Y)$.
\end{itemize}
\end{prop}

\begin{prop} {\rm (\cite[Proposition1.12]{CDG18})} Let $1<p<\infty.$
Suppose $\{w_{\delta}\}_{\delta>0}$  is a bounded sequence in $L^p(D)$.
Then the sequence $\{{\mathfrak T}_{\delta}(w_{\delta})\}_{\delta>0}$ is also bounded 
in $L^p(D\times Y)$. 
Furthermore, the following hold.
\begin{itemize}
\item[(i)]  If ${\mathfrak T}_{\delta}(w_{\delta})$ converges to some ${\bar w}$ weakly in 
$L^p(D\times Y)$, then $w_{\delta}$ converges to ${\mathcal M}_Y({\bar w})$  weakly in $L^p(D)$, 
where 
$$
{\mathcal M}_Y({\bar w})(x)=\int_Y {\bar w}(x,y)dy, \quad x\in D
$$
denotes the mean value operator on $L^p(D \times Y)$ into $L^p(D)$. 
\item[(ii)]  If $\{w_{\delta}\}$ converges to some $w$ weakly in $L^p(D)$, then up to a 
subsequence $\{\delta'\}$,  there exists $\hat{w}\in L^p(D\times Y)$ with 
${\mathcal M}_Y(\hat{w})=0$ such that ${\mathfrak T}_{\delta'}(w_{\delta'}) 
$ converges to $w+\hat{w}$ weakly in $L^p(D\times Y)$ and 
$$
\|w+\hat{w}\|_{L^p(D\times Y)} \le \liminf_{\delta' \to 0} \|w_{\delta'}\|_{L^p(D)}.
$$
\end{itemize}
\end{prop}

Let $(B, \|\cdot\|_B)$ be a (real) Banach space.  By definition, $L^p(D;B)$ for $1\le p<\infty$ is the set of functions 
$u:D\to B$ which are measurable and satisfy $\int_D \|u(x)\|_B^p dx <\infty$.

\begin{prop} {\rm (\cite[Corollary 1.37]{CDG18})} Let $1<p<\infty.$  Let $\{w_{\delta}\}_{\delta>0}$ be a 
sequence in $W^{1,p}(D)$ such that $w_{\delta}$ converges to some $w$  weakly in $W^{1,p}(D)$.
Then 
$$
{\mathfrak T}_{\delta}(w_{\delta}) \quad {\it converges \ to} \  w \ 
{\it  weakly \  in}  \ L^p(D;W^{1,p}(Y)).
$$
Furthermore, if $w_{\delta}$ converges to $w$  strongly in $L^p(D)$, then 
$$
{\mathfrak T}_{\delta}(w_{\delta}) \quad {\it converges \ to} \  w \ 
{\it strongly \  in}  \ L^p(D;W^{1,p}(Y)).
$$
\end{prop}

\begin{prop}  {\rm (\cite[Theorem 1.41]{CDG18})}  \label{prop:CDG18}
Suppose  $1<p<\infty$. Let $\{w_{\delta}\}$ 
be a sequence in $W^{1,p}(D)$ such that $w_{\delta}$ converges to some $w$ weakly in $W^{1,p}(D)$. 
Then for a subsequence, there exists some $\hat{w} \in L^p(D;{\mathcal W}^{1,p}_{\#}(Y))$ such that 
$$
{\mathfrak T}_{\delta}\Big(\frac{\partial w_{\delta}}{\partial x_i}\Big) \to \frac{\partial w}{\partial x_i}
 + \frac{\partial \hat{w}}{\partial y_i} \quad {\it weakly \ in} \ L^p(D\times Y), \ i=1,2,\ldots, d.
$$
\end{prop}

\section{Homogenization of the Dirichlet problem} \label{S:hom}

In this section, we prove the uniqueness of the limits and  consider the homogenized problem, which is satisfied 
by the limit $u_0$ obtained from the right after Theorem \ref{thm:exist}.

\begin{theorem} \label{thm1}
Assume {\bf (A1)$'$} and {\bf (A2)$'$} hold for the functions $A^\delta(x)=(a_{ij}^\delta(x))$,  
$B^\delta(x)=(b_i^\delta(x))$, $C^\delta(x)=(c_i^\delta(x))$ and $k^\delta(x)$ for $\delta>0$.  Let $u^{\delta}_{\lambda}$ be the solution of the problem \eqref{apriori1} for $\delta>0, \ 
\lambda>\beta_0$ and $f_{\delta}\in H^{-1}(D)$.  Suppose that the following conditions also hold:
\begin{itemize}
\item $\tilde{A}^\delta:={\mathfrak T}_\delta A^\delta$ converges to some matrix-valued function  $A$ a.e. on $D\times Y$; 
\item $\tilde{B}^\delta:={\mathfrak T}_\delta B^\delta$ converges to some vector-valued function  $B$ a.e. on $D\times Y$; 
\item $\tilde{C}^\delta:={\mathfrak T}_\delta C^\delta$ converges to some vector-valued function  $C$ a.e. on $D\times Y$;
\item $\tilde{k}^\delta:={\mathfrak T}_\delta k^\delta$ converges to some function  $k$ a.e. on  $D\times Y$,  and
\item $f_{\delta} \ \ {\it converges  \ to} \ \ f \ \ {\it strongly \ in} \ H^{-1}(D)$.
\end{itemize}
Then there exist  $u_0 \in H^1_0(D)$ and 
$u_1 \in L^2(D; {\mathcal H}_{\#}(Y))$ such that, as $\delta \to 0$, 
\begin{align*}
{(i)}  &  \quad  u^{\delta}_{\lambda}  \ \ {\it converges  \ to} \ \ u_0   \ \ {\it weakly \ in } \ \ H^1_0(D);  \\ 
{(ii)} &  \quad  u^{\delta}_{\lambda}   \ \ {\it converges  \ to} \ \ u_0   \ \ {\it strongly \ in } \ \ L^2(D); \\
{(iii)} &  \quad  {\mathfrak T}_{\delta}(\nabla_x  u^{\delta}_{\lambda})   \ \ {\it converges  \ to} \ \ 
\nabla_x u_0  + \nabla_y  u_1 \ \ \it strongly \ in  \ \ \big(L^2(D\times Y)\big)^d,  
\end{align*}
and the pair $(u_0, u_1)$ is the unique solution of the problem 
\[
 \!\!\!\!\! \iint\limits_{D\times Y} A(x,y) \Big(\nabla_x u_0(x)+ \nabla_y u_1(x,y)\Big) \cdot 
\Big(\nabla_x \varphi(x) + \nabla_y \psi (x,y)\Big) dxdy  
\]
\begin{equation} \label{eff-1} 
\begin{split}
&\quad  + \iint\limits_{D\times Y} B(x, y) \cdot \Big(\nabla_x u_0(x)+ \nabla_y u_1(x,y)\Big) \varphi(x)dxdy \\
&\quad + \iint\limits_{D\times Y}  u_0(x) \, C(x, y) \cdot \Big( \nabla_x \varphi(x) + \nabla_{y}\psi(x,y) \Big) dxdy  \\ 
  &\quad +\iint\limits_{D\times Y} u_0(x) \varphi(x)  k(x,y)dxdy  +\lambda(u_0, \varphi) \\
&= \langle f, \varphi \rangle_{H^{-1}, H^1_0} 
\end{split}
\end{equation}
 for $ \varphi \in H^1_0(D),  \ \psi \in L^2(D;{\mathcal H}_{\#}(Y))$. 
Furthermore, the following also hold:
\begin{align}  \label{eff-2}
 \lim_{\delta\to 0} \form^{\delta}_{\lambda}(u^{\delta}_{\lambda},u^{\delta}_{\lambda}) 
& = \int\limits_D A^{\sf eff}(x) \nabla u_0(x) \cdot \nabla u_0(x) dx + \int\limits_D B^{\sf eff}(x) \cdot \nabla u_0(x)  \, u_0(x)dx 
  \\ \nonumber
& \quad  + \int\limits_D u_0(x) \, C^{\sf eff}(x) \cdot \nabla u_0(x)dx   + \int_D u_0(x)^2 k^{\sf eff}(x) dx 
 +\lambda (u_0, u_0)  \\ \nonumber 
&= \form^0 (u_0, u_0)+{\lambda} (u_0, u_0)   =  \form^0_{\lambda} (u_0, u_0)
\end{align}
and  as $\delta \to 0$,  we have 
\begin{equation}  \label{eff-3}
{\mathfrak T}_{\delta}(\nabla u_{\lambda}^{\delta})(x,y) \ \ \mbox{converges to} \ \ 
\nabla_x u_0(x)+ \nabla_y u_1(x,y) \ \ {\it strongly} \ {\rm in} \ \big(L^2(D\times Y)\big)^d 
\end{equation}
Here the matrix-valued function $A^{\sf eff}(x)=(a^{\sf eff}_{ij}(x))$,  the vector-valued functions 
 $B^{\sf eff}(x)=(b_i^{\sf eff}(x))$ and $C^{\sf eff}(x)=(c^{\sf eff}_i(x))$ and the function $k^{\sf eff}(x)$ 
 are given  by
\begin{equation}\label{effcoef}
\begin{split}
&a_{ij}^{\sf eff}(x) = \int_Y \Big(a_{ij}(x, y) +\sum_{\ell=1}^d a_{i\ell}(x, y) \frac{\partial \omega_j}{\partial y_\ell}(x, y)\Big)
dy,  \\
& b_i^{\sf eff}(x)   = \int_Y \Big(b_i(x, y)+ \sum_{\ell=1}^d b_{l}(x, y)
\frac{\partial \omega_i}{\partial y_\ell}(x, y) \Big) dy, \\
&c_i^{\sf eff}(x)   = \int_Y\Big(\sum_{\ell=1}^d a_{i\ell}(x, y) \frac{\partial \omega_0}{\partial y_\ell}(x, y) +c_i(x, y)\Big)
dy, \\
&k^{\sf eff}(x) =\int_Y 
\Big(k(x, y)+ \sum_{\ell=1}^d b_\ell(x, y) \frac{\partial \omega_0}{\partial y_\ell}(x, y) \Big)dy,
\end{split}
\end{equation}
where $A(x,y)=(a_{ij}(x,y))$,  $B(x,y)=(b_i(x,y))$, $C(x,y)=(c_i(x,y))$ and $k(x,y)$ are the limit functions in the assumptions and 
$\{ \omega_i\}_{i=0}^d$ are the {\it correctors} defined in \eqref{effect} and \eqref{effect1}. 
\end{theorem}

\begin{remark}
Let us observe the following. 
\begin{itemize} 
\item[(1)]  Assumptions in Theorem \ref{thm1} imply that  \eqref{corre1} holds for $A$ and 
\eqref{corre2} holds for both $B$ and $C$.

\item[(2)] (periodic homogenization) When $A^\delta, B^\delta, C^\delta$ and $k^\delta$ are the forms 
\[
A^\delta(x)=A\Big(\frac x{\delta}\Big),  \qquad
B^\delta(x)=B\Big(\frac x{\delta}\Big), \qquad
C^\delta(x)=C\Big(\frac x{\delta}\Big),  \qquad  
k^\delta(x)=k\Big(\frac x{\delta}\Big),
\] 
for $Y$-periodic functions $A, B, C$ and $k$,  then  the limits are given by 
\[
A(x,y)=A(y),  \qquad
B(x,y)=B(y),  \qquad
C(x,y)=C(y),  \qquad
k(x,y)=k(y),
\]
respectively.  When
\[
A^\delta(x)=A_1(x) A_2(x/\delta), \qquad 
C^\delta(x)=c_1(x)  C_2(x/\delta), 
\]
where $A_2$ and $C_2$ are $Y$-periodic functions satisfying that  
both $A_1$ and $A_2$ satisfy {\bf (A1)} and  $C_2$ satisfies {\bf (A2)},  
and the function $c_1$ belongs to $L^{\infty}_{\sf loc}(\real^d)$, then the limits are given by 
$$
A(x,y)=A_1(x) A_2(y), \qquad
C(x,y)=c_1(x)  C_2(y)
$$ 
so that $A(x,y)=A_1(x)A_2(y)$ satisfies  \eqref{corre1} and  $C(x,y)=c_1(x)C_2(y)$ satisfies \eqref{corre2} 
(c.f.,   \cite[Remark 3.6]{CDG18}).
\end{itemize}
\end{remark}

\begin{proof}[ Proof of Theorem \ref{thm1}]
The existence of $(u_0, u_1)$ satisfying {(i)-(iii)} follows 
from Theorem \ref{thm:exist}  and Proposition \ref{prop:CDG18} (up to some subsequence at this point).

\medskip
{\bf Step 1.}  \ We first reveal the characterization of 
the limit ``$\lim_{\delta\to 0} \form_{\lambda}^{\delta}(u_{\lambda}^{\delta}, \varphi)$'' in the expression \eqref{apriori1} 
(up to a subsequence) for each $\varphi \in H_0^1(D)$ by using unfolding operators.
\begin{align}  \label{weak1} \nonumber
 & \langle f, \varphi \rangle_{H^{-1}(D), H^1_0(D)}  = \lim_{\delta \to 0} \langle f_{\delta}, \varphi \rangle_{H^{-1}(D), H^1_0(D)}  \ 
 = \lim_{\delta \to 0} \form^{\delta}_{\lambda} (u^{\delta}_{\lambda}, \varphi)   \\  \nonumber
& = \lim_{\delta \to 0}  \Big(  \int_D A^{\delta}(x) \nabla u^{\delta}_{\lambda}(x) \cdot  \nabla \varphi(x)dx 
+ \int_D B^{\delta}(x) \cdot  \nabla u^{\delta}_{\lambda} \, \varphi(x)dx \\  \nonumber
 \end{align}
\[
\begin{split}
& \qquad + \int_D u^{\delta}_{\lambda}  C^{\delta}(x) \cdot  \nabla \varphi(x)dx 
+ \int_D u_\lambda^\delta(x) \varphi(x) k^\delta(x)dx  + \lambda (u_{\lambda}^{\delta} ,\varphi) \Big) \\  \nonumber
& = \lim\limits_{\delta \to 0} \Big( \iint\limits_{D\times Y} {\mathfrak T}_{\delta} \big(A^{\delta} \nabla u^{\delta}_{\lambda} 
\cdot  \nabla \varphi\big)(x,y) dxdy + \iint\limits_{D\times Y}  {\mathfrak T}_{\delta}
 \big( B^{\delta} \cdot  \nabla u^{\delta}_{\lambda} \, \varphi\big)(x,y)dxdy    \\ 
& \qquad + \iint\limits_{D\times Y}  {\mathfrak T}_{\delta} 
 \big( u^{\delta}_{\lambda} C^{\delta} \cdot  \nabla \varphi\big)(x,y)dxdy 
+   \iint\limits_{D\times Y} {\mathfrak T}_\delta \big(u_\lambda^\delta \varphi k^\delta \big) (x,y)dxdy
    + \lambda(u_{\lambda}^{\delta}, \varphi) \Big)   \\ \nonumber
 & = \lim\limits_{\delta \to 0} \Big( \iint\limits_{D\times Y} \tilde{A}^\delta(x, y)  {\mathfrak T}_{\delta}(\nabla u^{\delta}_{\lambda})(x,y) \cdot  {\mathfrak T}_{\delta}(\nabla \varphi)(x,y) dxdy  \\ \nonumber
 & \qquad + \iint\limits_{D\times Y}  \tilde{B}^\delta(x, y)  \cdot  {\mathfrak T}_{\delta}(\nabla u^{\delta}_{\lambda})(x,y) {\mathfrak T}_{\delta}(\varphi\big)(x,y)dxdy  \\ \nonumber
& \qquad + \iint\limits_{D\times Y}  {\mathfrak T}_{\delta}(u^{\delta}_{\lambda}\big)(x,y) \tilde{C}^\delta(x, y)  \cdot  {\mathfrak T}_{\delta}(\nabla \varphi)(x,y) \,dxdy   \\ \nonumber 
 & \qquad  + \iint\limits_{D\times Y}  {\mathfrak T}_{\delta}(u^{\delta}_{\lambda}\big)(x,y)  {\mathfrak T}_{\delta}(\varphi \big)(x,y) 
 \tilde{k}^\delta(x, y) dxdy \Big)  + \lambda (u_0,\varphi)  \\ \nonumber 
& = \lim_{\delta \to 0} \Big( {\sf (I)}_\delta +{\sf (II)}_\delta + {\sf (III)}_\delta  +{\sf (IV)}_\delta \Big) + \lambda (u_0,\varphi).
\end{split}
\]
First we consider the term ${\sf (I)}_\delta$.  Since $A(x)=(a_{ij}(x))$ is bounded,  $\tilde{A}^\delta(x,y)$ is bounded and 
converges a.e.  Then the facts that ${\mathfrak T}_{\delta}(\nabla \varphi)$ converges to $\nabla \varphi$ 
{\it strongly} in $L^2(D \times Y)$ and  ${\mathfrak T}_{\delta}(\nabla u^{\delta}_{\lambda})$ converges to 
$\nabla_x  u_0 + \nabla_y u_1$  {\it weakly}  in $\big(L^2(D\times Y)\big)^d$ by {(iii)} imply that ${\sf (I)}_\delta$ 
goes to 
$$
\iint_{D\times Y} A(x, y) \Big(\nabla_x u_0(x) +\nabla_y u_1(x,y)\Big) \cdot \nabla \varphi (x)dxdy.
$$
To consider the term ${\sf (II)}_\delta$, note first that $\tilde{b}_i^\delta(x,y) {\mathfrak T}_\delta( \varphi)(x,y)$ 
converges to $b_i(x,y) \varphi(x)$ {\it strongly} in  $L^2(D\times Y)$, where 
$\tilde{B}^\delta(x,y)=(\tilde{b}_i^\delta(x,y))_{1\le i \le d}$ and $B(x,y)=(b_i(x,y))_{1\le i \le d}$.  In fact,  since 
$\varphi \in C_0^{\infty}(D) \subset L^{2p_0/(p_0-2)}(D)$ and ${\mathfrak T}_{\delta}(\varphi)$ converges to
$\varphi$ {\it strongly} in $L^{2p_0/(p_0-2)}(D\times Y)$ by Proposition \ref{propconverge} {(i)},  the assumption on $\tilde{B}^\delta$ 
tells us that 
\begin{align*}
& \| \tilde{b}^\delta_i {\mathfrak T}_{\delta}(\varphi) -b_i \varphi \|_{L^2(D\times Y)} \\
  &\le   \| \tilde{b}^\delta_i ({\mathfrak T}_{\delta}(\varphi) - \varphi) \|_{L^2(D\times Y)}
 + \| (\tilde{b}^\delta_i -b_i) \varphi \|_{L^2(D\times Y)}  \\
 & \le \Big(\iint_{D\times Y} |\tilde{b}^\delta_i(x,y)|^{p_0} dxdy \Big)^{1/p_0} \Big(\iint_{D\times Y}
\big|{\mathfrak T}_{\delta}(\varphi)(x,y) - \varphi(x)\big|^{2p_0/(p_0-2)} dxdy \Big)^{(p_0-2)/(2p_0)} \\
& \quad + \Big(\iint_{D\times Y} |\tilde{b}^\delta_i(x,y)-b_i(x,y)|^{p_0} dxdy \Big)^{1/p_0} \Big(\iint_{D\times Y}
\big|\varphi(x)\big|^{2p_0/(p_0-2)} dxdy \Big)^{(p_0-2)/(2p_0)}  \\ 
& \le \sup_{\delta>0} \|b_i^\delta\|_{L^{p_0}(D)} \| {\mathfrak T}_{\delta}(\varphi)- \varphi\|_{L^{2p_0/(p_0-2)}(D\times Y)}
 + \|\tilde{b}_i^\delta-b_i\|_{L^{p_0}(D\times Y)} \|\varphi\|_{L^{2p_0/(p_0-2)}(D)} 
\end{align*}  
which tends to $0$ as $ \delta \to 0$. 
So the term ${\sf (II)}_\delta$ goes to 
$$
\iint_{D\times Y} B(x,y)\cdot  \Big(\nabla_x u_0(x)+\nabla_y u_1(x,y) \Big) \varphi(x)dxdy.
$$
Similarly, we find that the third term ${\sf (III)}_\delta$ goes to
$$
\iint_{D\times Y} u_0(x) C(x, y) \cdot  \nabla \varphi(x)dxdy.
$$
As for the fourth term ${\sf (IV)}_\delta$,  since ${\mathfrak T}_\delta(\varphi)(x,y)\tilde{k}^\delta(x,y) $ converges to 
$\varphi(x)k(x,y)$ {\it strongly} in $L^2(D\times Y)$ in as much  as the similar estimate of the second term and 
${\mathfrak T}_\delta(u_\lambda^\delta)(x,y)$ also converges to $u_0$ {\it strongly} in $L^2(D\times Y)$, we 
see that the term ${\sf (IV)}_\delta$ goes to 
$$
\iint_{D\times Y} u_0(x) \varphi(x) k(x,y)dydx.
$$
Hence, summing up the limits, we see that 
\begin{align} \label{limit1}
\nonumber 
& \!\!\! \iint\limits_{D\times Y} A(x,y) \Big(\nabla_x u_0(x) +\nabla_y u_1(x,y)\Big)  \cdot  \nabla \varphi (x)dxdy  \\ \nonumber
&\quad  + \iint\limits_{D\times Y} B(x,y) \cdot  \Big(\nabla_x u_0(x)+\nabla_y u_1(x,y) \Big) \varphi(x)dxdy \\ \nonumber
& \quad +\iint\limits_{D\times Y} u_0(x) C(x, y) \cdot  \nabla \varphi(x)dxdy 
+\iint\limits_{D\times Y} u_0(x) \varphi(x) k(x, y)dxdy  + \lambda (u_0, \varphi) \\
& = \langle f, \varphi \rangle_{H^{-1}, H^1_0}, \quad \varphi \in H_0^1(D).
\end{align}

\medskip
{\bf Step 2.} \  In this step, we show that the  pair $(u_0, u_1)$ is a unique solution of \eqref{eff-1} 
(and this then gives us the full sequences converge to the limits respectively in {(i)-(iii)}). 
To this end, we note here that the limits $A, B, C$, and $k$ satisfy the following because of the 
conditions {\bf (A1)$'$}, {\bf (A2)$'$} and the assumptions in the theorem:
\begin{itemize}
\item[{\bf (A1)$''$}] $\dis \alpha |\xi|^2 \le A(x,y) \xi \cdot \xi \le \beta |\xi|^2$ for 
$\xi \in \real^2, \ {\rm a.e.} \ (x,y)\in D\times Y$; \medskip
\item[{\bf (A2)$''$}] $B, C\in (L^{p_0}(D\times Y))^d$ and $k \in L^{p_0/2}(D\times Y)$. 
\end{itemize}
Now for any $\theta_1\in C_0^{\infty}(D)$ and $\theta_2 \in C_{\#}^{\infty}(Y)$, 
set 
$$
v_{\delta}(x):= \delta \theta_1(x) \theta_2\Big(\frac x{\delta}\Big), \quad x \in D.
$$  
Then $v_{\delta}$ belongs to $H_0^1(D)$ for $\delta>0$ and the sequence $\{v_{\delta}\}$ converges to $0$  
uniformly in $D$ as $\delta \to 0$.  Moreover, their gradients 
$$
\nabla v_{\delta}(x)= \delta \, \nabla \theta_1(x)  \theta_2 \Big(\frac x{\delta}\Big)
+ \theta_1(x) \big(\nabla \theta_2\big)\Big(\frac x{\delta}\Big)
$$
are bounded in $\big(L^2(D)\big)^d$.  Therefore $v_{\delta}$ converges to $0$ {\it weakly} in $H^1_0(D)$.
On the other hand, by the definition of the unfolding operators, we see that
$$
{\mathfrak T}_{\delta}\big( \nabla v_{\delta}\big)(x,y)
=\delta {\mathfrak T}_{\delta}\big( \nabla \theta_1\big)(x,y)  \theta_2 (y) +
{\mathfrak T}_{\delta}(\theta_1)(x,y) \nabla_y \theta_2(y)
$$
converges to  $\theta_1 (x)  \nabla_y \theta_2(y)$ {\it strongly} in  $\big(L^2(D\times Y)\big)^d$.  
Thus taking $v_{\delta}$ as a test functions in \eqref{apriori1} and passing $\delta \to 0$, we find that
\begin{align*}
0   &   = \lim_{\delta \to 0} \langle f_{\delta}, v_{\delta} \rangle  
\  = \lim_{\delta \to 0} \form^{\delta}_{\lambda}(u^{\delta}_{\lambda}, v_{\delta})  \\
& = \lim_{\delta \to 0} \Big(\int\limits_D A^{\delta}(x) \nabla u^{\delta}_{\lambda}(x)  \cdot  \nabla v_{\delta} (x)dx 
+ \int\limits_D B^{\delta}(x) \cdot  \nabla u^{\delta}_{\lambda}(x) v_{\delta}(x)dx  \\
& \qquad  + \int\limits_D u^{\delta}_{\lambda}(x) C^{\delta}(x)  \cdot  \nabla v_{\delta}(x)dx  
 + \int\limits_D u_\lambda^\delta(x) v_\delta(x) k^\delta(x)dx
+\lambda(u^{\delta}_{\lambda}, v_{\delta}) \Big)  \\
& = \lim_{\delta \to 0} \Big(\iint\limits_{D\times Y} {\mathfrak T}_{\delta}\Big(A^{\delta} \nabla u^{\delta}_{\lambda}
 \cdot  \nabla v_{\delta} \Big)(x,y) dxdy + \iint\limits_{D\times Y} 
  {\mathfrak T}_{\delta}\Big( B^{\delta} \cdot  \nabla u^{\delta}_{\lambda} v_{\delta}\Big)(x,y)dxdy \\
 &\qquad  + \iint\limits_{D\times Y} {\mathfrak T}_{\delta}\Big(u^{\delta}_{\lambda} C^{\delta} \cdot  \nabla v_{\delta}\Big)(x,y)dxdy  
 + \iint\limits_{D\times Y} {\mathfrak T}_{\delta}\Big(u_\lambda^\delta v_\delta k^\delta \Big) dxdy \Big) \\
& = \lim_{\delta \to 0} \Big(\iint\limits_{D\times Y} \tilde{A}^\delta(x, y) {\mathfrak T}_{\delta} \big(\nabla u^{\delta}_{\lambda}\big)(x,y) 
\cdot \Big( \delta {\mathfrak T}_{\delta}\big(\nabla_x \theta_1\big)(x,y) \, \theta_2(y)  +  {\mathfrak T}_{\delta}(\theta_1)(x,y) \nabla_y \theta_2(y)\Big)  dxdy  \\
& \qquad   + \delta \iint\limits_{D\times Y} \tilde{B}^\delta(x, y) 
\cdot {\mathfrak T}_{\delta}(\nabla u^{\delta}_{\lambda})(x,y) \, {\mathfrak T}_{\delta}(\theta_1)(x,y) \theta_2(y) dxdy   \\ 
& \qquad  + \iint\limits_{D\times Y}  {\mathfrak T}_{\delta}(u^{\delta}_{\lambda})(x,y) \tilde{C}^\delta(x, y) \cdot  
\Big( \delta {\mathfrak T}_{\delta}(\nabla \theta_1)(x,y) \theta_2(y)  +  {\mathfrak T}_{\delta}(\theta_1)(x,y) 
\nabla_y \theta_2(y) \Big) dxdy  \\ 
& \qquad  + \delta \iint\limits_{D\times Y}  {\mathfrak T}_{\delta}(u^{\delta}_{\lambda})(x,y)
{\mathfrak T}_{\delta}(\theta_1)(x,y) \theta_2(y) \tilde{k}^\delta(x, y)dxdy \Big) \\
& = \iint\limits_{D\times Y} A(x,y) \Big(\nabla_x u_0(x) +\nabla_y u_1(x,y)\Big) \cdot \nabla_y \theta_2(y) \theta_1(x)dxdy \\
& \qquad + \iint\limits_{D\times Y} u_0(x) C(x,y)\cdot   \nabla_y \theta_2(y) \theta_1(x) dxdy. 
\end{align*}
So, since $C_0^{\infty}(D)\times C_{\#}^{\infty}(Y)$ is dense in $L^2(D;{\mathcal H}_{\#}(Y))$, it follows that, 
for $\psi \in L^2(D;{\mathcal H}_{\#}(Y))$, 
\[
 \iint\limits_{D\times Y} A(x,y)\Big(\nabla u_0(x)+\nabla_y u_1(x,y)\Big)\cdot \nabla_y \psi(x,y) dxdy  +\iint\limits_{D\times Y} u_0(x) C(x,y) \cdot \nabla_y \psi(x,y) dxdy=0 
\]
and, summing up this with \eqref{limit1}, we find that $(u_0, u_1)$ is a solution of \eqref{eff-1}.
Its uniqueness property is then proved by using the standard 
variational form in the space
$$
{\mathfrak H}:= H_0^1(D) \times L^2(D;{\mathcal H}_{\#}(Y))
$$
endowed with the norm 
$$
\|(\varphi, \psi)\|_{\mathfrak H}:= \Big( \int_D \nabla \varphi(x) \cdot \nabla \varphi(x)dx 
+ \int_D  \Big(\int_Y \nabla_y \psi(x,y) \cdot  \nabla_y \psi(x,y) dy\Big) dx \Big)^{1/2}.
$$
Then,  we solve the following variational problem:
Find $(u_0,u_1) \in {\mathfrak H}$ such that 
\begin{align}  \label{vi}  \nonumber 
& \hspace*{-0.5cm}
\iint\limits_{D\times Y} A(x,y)\Big(\nabla_x u_0(x)+ \nabla_y u_1(x,y) \Big)\cdot 
\Big(\nabla \varphi(x)+\nabla_y \psi(x,y)\Big) dxdy \\ \nonumber 
&  + \iint\limits_{D\times Y} B(x, y) \cdot \Big(\nabla_x u_0(x)  + \nabla_y u_1(x,y) \Big)\varphi(x)dxdy  \\ \nonumber
& + \iint\limits_{D\times Y} u_0(x) C(x, y) \cdot \Big(\nabla_x \varphi(x)+\nabla_y \psi(x,y)\Big) dxdy   \\ 
&   + \iint\limits_{D\times Y} u_0(x)\varphi(x) k(x, y)dxdy + \lambda (u_0, \varphi) \  =\langle f, \varphi \rangle 
\end{align}
holds for all $(\varphi,\psi) \in {\mathfrak H}$.

To this end,  define a bilinear functional on ${\mathfrak H} \times {\mathfrak H}$:
\begin{align*}
& \!\!\!\!\! \eta((\varphi_0, \psi_0),(\varphi_1, \psi_1)) \\
 & :=\iint\limits_{D\times Y} A(x,y)\big(\nabla_x \varphi_0(x)+ \nabla_y \psi_0(x,y)\big)\cdot 
\big(\nabla_x \varphi_1(x)+ \nabla_y \psi_1(x,y)\big) dxdy \\
& \qquad + \iint\limits_{D\times Y} B(x, y)\cdot \Big(\nabla_x \varphi_0(x)+ \nabla_y \psi_0(x,y) \Big) 
\varphi_1(x)dxdy  \\
& \qquad + \iint\limits_{D\times Y} \varphi_0(x) C(x, y)\cdot \Big(\nabla_x \varphi_1(x)+\nabla_y \psi_1(x,y)\Big)dxdy \\
& \qquad +  \iint\limits_{D\times Y} \varphi_0(x)  \varphi_1(x) k(x, y) dxdy   +\lambda (\varphi_0, \varphi_1) 
\end{align*}
which satisfies  the estimate
\begin{align*}
&  \big|\eta((\varphi_0, \psi_0),(\varphi_1, \psi_1))\big|  \\
& \le \beta \big( \|\nabla_x \varphi_0\|_{L^2(D)} + \|\nabla_y\psi_0\|_{L^2(D\times Y)} \big)
\big( \|\nabla_x \varphi_1\|_{L^2(D)} + \|\nabla_y\psi_1\|_{L^2(D\times Y)} \big) \\
& \ \  +  \big(\sum_{i=1}^d \|b_i\|_{L^{p_0}(D\times Y)}\big) \|\varphi_1\|_{L^2(D)}
\Big( \|\nabla_x \varphi_0\|_{L^2(D)}  + \|\nabla_y \psi_0\|_{L^2(D\times Y)} \Big)   \\
&  \ \ + \big(\sum_{i=1}^d \|c_i\|_{L^{p_0(D\times Y)}} \big) \|\varphi_0\|_{L^2(D)} 
\big(\| \nabla_x \varphi_1\|_{L^2(D)} + \|\nabla_y\psi_0\|_{L^2(D\times Y)} \big)  \\
& \ \  + \| k\|_{L^{p_0/2}(D\times Y)}  \| \varphi_0\|_{L^2(D)}\| \varphi_1\|_{L^2(D)}
 + \lambda \|\varphi_0\|_{L^2(D)} \|\varphi_1\|_{L^2(D)}   \\
& \le \Big(\beta +  c_1\sum_{i=1}^d \big(\|b_i\|_{L^{p_0}(D\times Y)}+\|c_i\|_{L^{p_0}(D\times Y)}\big) 
+  c_1 \| k\|_{L^{p_0/2}(D\times Y)} + c_1\lambda \Big) \\
&\qquad \Big( \|(\varphi_0,\psi_0)\|_{\mathfrak H} \times \|(\varphi_1,\psi_1)\|_{\mathfrak H} \Big)
\end{align*}
for some constant $c_1>0$ by {\bf (A1)$''$} and {\bf (A2)$''$}. 
This shows that $\eta$ is bounded on $\mathfrak H$.   
Moreover  by using Lemma  \ref{lem:est1} and the inequality \eqref{comparison},  we find some constant 
$c_2>0$ such that 
\begin{align*}
&\!\!\!\!\! \eta((\varphi_0, \psi_0),(\varphi_0, \psi_0)) \\
&  =\iint_{D\times Y} A(x,y)\big(\nabla_x \varphi_0(x)+ \nabla_y \psi_0(x,y)\big) \cdot 
\big(\nabla_x \varphi_0(x) + \nabla_y \psi_0(x,y)\big) dxdy \\
& \qquad + \iint\limits_{D\times Y} B(x, y)\cdot \big(\nabla_x \varphi_0(x)+ \nabla_y \psi_0(x,y) \big)\varphi_0(x) dxdy  \\
& \qquad   + \iint\limits_{D\times Y} \varphi_0(x) C(x, y) \cdot \big( \nabla_x \varphi_0(x) +\nabla_y \psi_0(x,y) \big) dxdy  \\
& \qquad +  \iint\limits_{D\times Y} \varphi_0(x)^2 k(x, y)dxdy + \lambda (\varphi_0, \varphi_0) \\
& \ge \alpha \iint_{D\times Y} | \nabla_x \varphi_0(x)+\nabla_y \psi_0(x,y)|^2 dxdy \\
& \qquad - \Big|\iint\limits_{D\times Y} 
B(x, y)\cdot \big(\nabla_x \varphi_0(x)+ \nabla_y \psi_0(x,y) \big) \varphi_0(x)dxdy \Big|   \\
& \qquad - \Big| \iint\limits_{D\times Y} \varphi_0(x) C(x, y) \cdot \big(\nabla_x \varphi_0(x) 
+ \nabla_y \psi_0(x,y) \big) dxdy \Big|  \\ 
& \qquad + \iint\limits_{D\times Y} \varphi_0(x)^2 k(x, y)dxdy + \lambda (\varphi_0, \varphi_0)   \\ 
&  = \alpha \Big( \int_D\nabla \varphi_0(x) \cdot \nabla \varphi_0(x) dx + \iint_{D\times Y} \nabla_y \psi_0(x,y) \cdot 
\nabla_y \psi_0(x,y) dxdy\Big) \\ 
& \qquad - \Big|\iint\limits_{D\times Y} B(x, y)\cdot \big(\nabla_x \varphi_0(x) 
+ \nabla_y \psi_0(x,y) \big)\varphi_0(x) dxdy\Big|  \\
& \qquad - \Big| \iint\limits_{D\times Y} \varphi_0(x) 
C(x, y) \cdot \big(\nabla_x \varphi_0(x)+ \nabla_y \psi_0(x,y) \big) dxdy \Big|   \\
& \qquad  +\iint\limits_{D\times Y} \varphi_0(x)^2 k(x, y)dxdy + \lambda (\varphi_0, \varphi_0)   \\ 
& \ge c_2 \Big( \int_D\nabla \varphi_0(x) \cdot  \nabla \varphi_0(x) dx + \iint_{D\times Y} \nabla_y \psi_0(x,y) 
\cdot  \nabla_y \psi_0(x,y) dxdy\Big) \\ 
& = c_2  \| (\varphi_0, \psi_0)\|_{\mathfrak H}^2,
\end{align*}
since 
$$
\iint_{D\times Y} \nabla_x \varphi_0(x) \cdot  \nabla_y \psi_0(x,y) dxdy
 = \int_D \nabla_x \varphi_0(x) \cdot  \Big( \int_Y \nabla_y \psi_0(x,y)dy\Big)dx=0.
$$
This shows $\eta$ is a coercive bilinear form on ${\mathfrak H}$. Thus, using again the  Lax-Milgram theorem,  we find a unique $(u_0, u_1) \in {\mathfrak H}$ so that 
\eqref{vi} holds, in other words, the pair $(u_0,u_1)$ is the unique solution of \eqref{eff-1}.

\medskip
{\bf Step 3.} In this step,  we characterize the function $u_1$.  Plugging $\varphi\equiv 0$ in \eqref{vi}, we see that 
\begin{equation} \label{vi0}
\begin{split}
 & \iint\limits_{D\times Y} A(x,y)\Big(\nabla_x u_0(x)+\nabla_y u_1(x,y)\Big)\cdot  \nabla_y \psi(x,y)dxdy  \\
 &\quad + \iint\limits_{D\times Y}  u_0(x) C(x, y) \cdot \nabla_y \psi(x,y)dxdy  =0  
\end{split}
\end{equation}
for all $\psi \in L^2(D;{\mathcal H}_{\#}(Y))$.  Then inserting $\psi(x,y)=\theta_1(x) \theta_2(y)$ 
for $\theta_1 \in C_0^\infty(D)$ and $\theta_2 \in {\mathcal H}_{\#}(Y)$ in \eqref{vi0}, we see
\[
\begin{split}
 \int_D \theta_1(x) \Big(\int_Y  A(x, y) \Big(\nabla_x u_0(x) +\nabla_y u_1(x,y) 
\Big)\cdot  \nabla_y \theta_2(y)dy  
+ \int_Y u_0(x) C(x, y) \cdot  \nabla \theta_2(y)dy \Big)dx =0.
\end{split}
\]
Thus,  the equation 
\begin{equation} \label{vi2}
\begin{split}
 & \int\limits_Y A(x, y) \Big(\nabla_x u_0(x) +\nabla_y u_1(x,y) \Big) \cdot  \nabla_y \psi(y)dy \\
 & \quad + \int\limits_Y u_0(x) C(x, y) \cdot \nabla_y \psi(y)dy  = 0  \quad   {\rm a.e. \ on} \ D
\end{split}
\end{equation}
holds for $\psi \in {\mathcal H}_{\#}(Y)$.    On the other hand, the equations \eqref{effect} and \eqref{effect1} imply 
 that 
\begin{equation} \label{corr2}
\begin{array}{l}
\dis \int_Y A(x, y) \Big(\xi + \nabla_y \phi_{\xi}(x, y)\Big) \cdot  \nabla_y \psi(y) dy =0 \quad {\rm and} \qquad \\
 \ \vspace*{-8pt} \\
\dis \int_Y \Big(A(x, y) \nabla_y w_0(x, y) \cdot \nabla_y \psi(y) +  C(x, y) \cdot \nabla_y \psi (y)\Big)dy =0
\end{array}
\end{equation}
for all $\psi \in{\mathcal H}_{\#}(Y)$ and $\xi \in \rd$ with $\phi_{\xi}(x, y)= \sum_{i=1}^d \xi_i \omega_i(x, y)$.
Now, putting $\xi=\nabla_x u_0(x)$  in the first equation of \eqref{corr2},  multiplying $u_0(x)$ 
in the second equation of \eqref{corr2} and then summing up both sides, we find that 
\begin{align} \label{corr2-1}
 & \int_Y A(x, y)\Big(\nabla_x u_0(x) + \nabla_y \phi(x,y)  \Big) \cdot  \nabla_y \psi(y) dy  \\ \nonumber 
& \quad +\int_Y \Big( u_0(x) A(x, y) \nabla_y w_0(x,y) \cdot \nabla_y \psi(y) + u_0(x)  C(x, y) \cdot \nabla_y  \psi (y)\Big) dy =0  \nonumber
\end{align}
a.e.   $x\in D$ for all  $\psi \in {\mathcal H}_{\#}(Y)$, where 
$\phi(x,y)=\sum_{i=1}^d \frac{\partial u_0}{\partial x_i}(x) \omega_i(x, y)$ for $(x,y) \in D\times Y$.  
So subtracting then both sides of \eqref{vi2} and \eqref{corr2-1},  we have that 
$$
\int_Y A(x,y) \Big( \nabla_y u_1(x,y) -\nabla_y \phi(x,y) -u_0(x) \nabla_y w_0(x, y) \Big) \cdot \nabla_y \psi(y)dy=0, 
$$
a.e.  on  $D$,   for all $\psi \in{\mathcal H}_{\#}(Y)$.
Therefore, we can conclude that 
\begin{equation} \label{corr3}
\begin{split}
u_1(x,y) &=u_0(x) w_0(x, y)+ \phi(x,y) =u_0(x) w_0(x, y)+ \sum_{i=1}^d \frac{\partial u_0}{\partial x_i}(x)\omega_i(x, y). 
\end{split}
\end{equation}

\medskip
{\bf Step 4.} \ We now consider the problem for which the $(u_0, u_1)$ solves.  Using the characterization 
of $u_1$ in the previous step \eqref{corr3}, we see from \eqref{vi} that,  
\begin{align*}
& \! \iint\limits_{D\times Y} A(x,y) \Big(\nabla_x u_0(x) +\nabla_y u_1(x,y)\Big) \cdot \nabla_x \varphi(x)dxdy  \\
& \quad  + \iint\limits_{D\times Y} B(x,y) \cdot  \Big(\nabla_x u_0(x) +\nabla_y u_1(x,y)\Big) \varphi(x) dxdy \\ 
& \quad + \iint\limits_{D\times Y} u_0(x) C(x,y) \cdot  \nabla_x \varphi(x)dxdy   + \iint\limits_{D\times Y} u_0(x) \varphi(x) k(x, y) dxdy  \\
& = \iint\limits_{D\times Y} A(x,y) \Big(\nabla_x u_0(x) + u_0(x) \nabla_y \omega_0(x,y) + \nabla_y \phi(x,y) \Big) \cdot \nabla_x \varphi(x)dxdy   \\
& \quad  + \iint\limits_{D\times Y} B(x,y)\cdot  \Big(\nabla_x u_0(x) + u_0(x) \nabla_y \omega_0(x,y) + \nabla_y \phi(x,y) \Big)  \varphi(x) dxdy \\
& \quad +  \iint\limits_{D\times Y} u_0(x) C(x,y) \cdot  \nabla_x \varphi(x)dxdy 
+ \iint\limits_{D\times Y} u_0(x) \varphi(x) k(x, y) dxdy  \\
& = \iint\limits_{D\times Y} A(x,y) \Big(\nabla_x u_0(x) + \nabla_y \phi(x,y) \Big) \cdot \nabla_x \varphi(x)dxdy  \\
& \quad  + \iint\limits_{D\times Y} B(x,y)\cdot  \Big(\nabla_x u_0(x)  + \nabla_y \phi(x,y) \Big)  \varphi(x) dxdy  \\
& \quad +   \iint\limits_{D\times Y} u_0(x) \Big(A(x,y) \nabla_y \omega_0(x,y) +  C(x,y)\Big) \cdot  \nabla_x \varphi(x)dxdy  \\
& \quad +  \iint\limits_{D\times Y} u_0(x) \Big( k(x, y)+ B(x,y) \cdot \nabla_y \omega_0 (x,y) \Big) \varphi(x) dxdy  \\
& = \int_D A^{\sf eff}(x) \nabla u_0(x) \cdot \nabla \varphi(x) dx 
+ \int_D B^{\sf eff}(x)\cdot \nabla u_0(x) \, \varphi(x) dx  \\ 
& \quad + \int_D u_0(x) \, C^{\sf eff}(x) \cdot \nabla \varphi (x)dx  +  \int_D u_0(x) \varphi(x) k^{\sf eff}(x)dx
\end{align*}
for any $\varphi \in H_0^1(D)$, where 
$A^{\sf eff}(x)=(a_{ij}^{\sf eff}(x))$ is a $d\times d$-matrix-valued function, $B^{\sf eff}(x)=(b_i^{\sf eff}(x))$ and 
$C^{\sf eff}(x)=(c_i^{\sf eff}(x))$ are $d$-dimensional vector-valued function and  $k^{\sf eff}(x)$ is a function, are defined in  \eqref{effcoef}. Thus the equation \eqref{eff-1} induces
\begin{align*}
& \int_D A^{\sf eff} \nabla u_0(x) \cdot \nabla \varphi(x) dx + \int_D B^{\sf eff}\cdot \nabla u_0(x) \, \varphi(x) dx 
+ \int_D u_0(x)  C^{\sf eff}\cdot \nabla \varphi(x) dx  \\
& \quad  +  \int_D u_0(x)\varphi(x) k^{\sf eff}(x)dx +\lambda (u_0, \varphi)  =\langle f, \varphi \rangle,  
\qquad  \varphi \in H_0^1(D).
\end{align*}
This also shows that the limit \eqref{eff-2} holds  by noting that $u_{\lambda}^{\delta}$ converges to 
$u_0$  {\it strongly} in $L^2(D)$ and {\it weakly} in $H^1_0(D)$ (therefore  $\nabla u_{\lambda}^{\delta}$ 
converges to $\nabla u_0$ {\it weakly} in $\big(L^2(D)\big)^d$).
 
\medskip
{\bf Step 5.}  We finally show \eqref{eff-3}.
\begin{align*} \nonumber
& \alpha \iint\limits_{D\times Y} \big|
{\mathfrak T}_{\delta}\big(\nabla u^{\delta}_{\lambda}\big) 
-\nabla_x u_0 - \nabla_y u_1\big|^2 dxdy \\  \nonumber
& \le  \iint\limits_{D\times Y} 
A(x,y)\Big({\mathfrak T}_{\delta} \big(\nabla u^{\delta}_{\lambda}\big)-
\nabla_x u_0 - \nabla_y u_1\Big) \cdot  \Big({\mathfrak T}_{\delta}\big(\nabla u^{\delta}_{\lambda}\big) -
\nabla_x u_0 - \nabla_y u_1\Big)  dxdy \\  \nonumber
& =  \iint\limits_{D\times Y} 
A(x, y){\mathfrak T}_{\delta}\big(\nabla u^{\delta}_{\lambda}\big) \cdot  {\mathfrak T}_{\delta}\big(\nabla u^{\delta}_{\lambda}\big)dxdy  - \iint\limits_{D\times Y} A(x, y) {\mathfrak T}_{\delta}  \big(\nabla u^{\delta}_{\lambda}\big)  \cdot
\big(\nabla_x u_0+\nabla_y u_1\big)dxdy \\  \nonumber
& \qquad - \iint\limits_{D\times Y} A(x, y) \big(\nabla_x u_0+\nabla_y u_1\big)  \cdot 
{\mathfrak T}_{\delta}\big(\nabla u^{\delta}_{\lambda}\big)dxdy  \\
& \qquad + \iint\limits_{D\times Y} 
A(x, y) \big(\nabla_x u_0+\nabla_y u_1\big)  \cdot \big(\nabla_x u_0+\nabla_y u_1\big)dxdy \\ \nonumber
 &=  {\sf (V)}_\delta -\iint\limits_{D\times Y} B(x, y)  \cdot {\mathfrak T}_{\delta}\big(\nabla u^{\delta}_{\lambda}\big) 
{\mathfrak T}_{\delta}(u^{\delta}_{\lambda}) dxdy  - \iint\limits_{D\times Y}  {\mathfrak T}_{\delta} (u_\lambda^{\delta}) 
 C(x, y) \cdot  {\mathfrak T}_{\delta}(\nabla u_\lambda^{\delta})dxdy \\ 
 & \qquad - \iint\limits_{D\times Y} \big({\mathfrak T}_{\delta}(u_\lambda^{\delta})\big)^2 k(x,y) dxdy  
- \iint\limits_{D\times Y} A(x, y) {\mathfrak T}_{\delta}\big(\nabla u^{\delta}_{\lambda}\big) \cdot
\big(\nabla_x u_0+\nabla_y u_1\big)dxdy  \\ 
& \qquad  - \iint\limits_{D\times Y} A(x, y) \big(\nabla_x u_0+\nabla_y u_1\big) \cdot
{\mathfrak T}_{\delta}\big(\nabla u^{\delta}_{\lambda}\big)dxdy   \\
&\qquad + \iint\limits_{D\times Y} A(x, y) \big(\nabla_x u_0+\nabla_y u_1\big)  \cdot \big(\nabla_x u_0+\nabla_y u_1\big)dxdy,  \nonumber
\end{align*}
where
\[ 
\begin{split}
{\sf (V)}_\delta & = \iint\limits_{D\times Y} 
 A(x, y){\mathfrak T}_{\delta}\big(\nabla u^{\delta}_{\lambda}\big)  \cdot
{\mathfrak T}_{\delta}\big(\nabla u^{\delta}_{\lambda} \big)dxdy + \iint\limits_{D\times Y} 
B(x, y) \cdot {\mathfrak T}_{\delta}\big(\nabla u^{\delta}_{\lambda}\big) 
 {\mathfrak T}_{\delta}(u_{\delta})  dxdy \\
&\quad + \iint\limits_{D\times Y}  {\mathfrak T}_{\delta} (u_\lambda^{\delta}) 
 C(x, y)  \cdot {\mathfrak T}_{\delta}(\nabla u_\lambda^{\delta})dxdy + 
 \iint\limits_{D\times Y} \big({\mathfrak T}_{\delta}(u_\lambda^{\delta})\big)^2 k(x,y) dxdy.
 \end{split}
 \]
We estimate each term on the right-hand side.  Since $u^{\delta}_{\lambda}$ converges to 
$u_0$ {\it strongly} in $L^2(D)$, ${\mathfrak T}_{\delta}(u^{\delta}_{\lambda})$ 
(resp. $b_i {\mathfrak T}_{\delta}(u^{\delta}_{\lambda})$,  $c_i {\mathfrak T}_{\delta}(u^{\delta}_{\lambda})$ 
and  $k{\mathfrak T}_{\delta}(u^{\delta}_{\lambda})$) converges to $u_0$ 
(resp. $b_i u_0$,   $c_i u_0$ and $k u_0$) {\it strongly} in $L^2(D\times Y)$. Thus, noting that 
${\mathfrak T}_{\delta}(\nabla u^{\delta}_{\lambda})$ converges to $\nabla_x u_0+\nabla_y u_1$  
{\it weakly} in $\big(L^2(D\times Y)\big)^d$,  we see that 
$$
\lim_{\delta \to 0} 
\iint\limits_{D\times Y} B(x,y)  \cdot {\mathfrak T}_{\delta}\big(\nabla u^{\delta}_{\lambda}\big) 
 {\mathfrak T}_{\delta}(u^{\delta}_{\lambda})dxdy =
 \iint\limits_{D\times Y} B(x, y)  \cdot \big(\nabla_x u_0+\nabla_y u_1\big) u_0 dxdy, 
$$
\[
\begin{split}
\lim_{\delta \to 0} \iint\limits_{D\times Y} 
{\mathfrak T}_{\delta} (u_\lambda^{\delta}) 
 C(x, y)  \cdot {\mathfrak T}_{\delta}(\nabla u_\lambda^{\delta})dxdy &= 
 \iint\limits_{D\times Y} u_0 \, C(x, y)  \cdot \big(\nabla_x u_0 +\nabla_y u_1\big)dxdy  \\
 & = \iint\limits_{D\times Y} u_0 \, C(x, y)   \cdot \nabla_x u_0 dxdy, 
 \end{split}
\]
$$
\lim_{\delta \to 0} \iint\limits_{D\times Y} \big({\mathfrak T}_{\delta}(u_\lambda^{\delta})\big)^2 k(x, y) dxdy   
= \iint\limits_{D\times Y} u_0(x)^2 k(x, y)dxdy,
$$
\[
\begin{split}
& \lim_{\delta \to 0} 
\iint\limits_{D\times Y} A(x, y) {\mathfrak T}_{\delta}\big(\nabla u^{\delta}_{\lambda}\big) \cdot
\big(\nabla_x u_0+\nabla_y u_1\big)dxdy \\ 
& \qquad  =
\iint\limits_{D\times Y} A(x, y) \big(\nabla_x u_0+\nabla_y u_1\big) \cdot 
\big(\nabla_x u_0+\nabla_y u_1\big)dxdy
 \end{split}
\]
and 
\[
\begin{split}
& \lim_{\delta \to 0} 
\iint\limits_{D\times Y} A(x, y) \big(\nabla_x u_0+\nabla_y u_1\big) \cdot 
{\mathfrak T}_{\delta}\big(\nabla u^{\delta}_{\lambda}\big) dxdy  \\
& \qquad  =
\iint\limits_{D\times Y} A(x, y) \big(\nabla_x u_0+\nabla_y u_1\big) \cdot 
\big(\nabla_x u_0+\nabla_y u_1\big)dxdy.
 \end{split}
\]
Finally, the first term is estimated as follows:
\begin{align*}
\lim_{\delta \to 0} {\sf (V)}_\delta
& = \lim_{\delta \to 0}  \Big(\iint\limits_{D\times Y} 
A(x, y){\mathfrak T}_{\delta}\big(\nabla u^{\delta}_{\lambda} \big)  \cdot {\mathfrak T}_{\delta}
\big(\nabla u^{\delta}_{\lambda} \big)dxdy  + \iint\limits_{D\times Y} 
B(x, y)  \cdot {\mathfrak T}_{\delta}\big(\nabla u^{\delta}_{\lambda}\big) 
{\mathfrak T}_{\delta}(u^{\delta}_{\lambda})dxdy \\
 & \qquad +\iint\limits_{D\times Y} {\mathfrak T}_{\delta}(u^{\delta}_{\lambda}) \, 
 C(x, y)  \cdot {\mathfrak T}_{\delta}(\nabla u^{\delta}_{\lambda})dxdy  + \iint\limits_{D\times Y} \big({\mathfrak T}_{\delta}(u_\lambda^{\delta})\big)^2 k(x, y) dxdy \Big) \\
& = \lim_{\delta \to 0}  \Big(\iint\limits_{D\times Y} 
{\mathfrak T}_{\delta}\big( A^{\delta} \nabla u^{\delta}_{\lambda}  \cdot \nabla u^{\delta}_{\lambda}\big)dxdy  
 + \iint\limits_{D\times Y} {\mathfrak T}_{\delta}\big( B^{\delta} \cdot \nabla u^{\delta}_{\lambda} \,  
 u_{\delta}\big)dxdy  \\ 
 & \qquad + \iint\limits_{D\times Y} {\mathfrak T}_{\delta}\big(u_\lambda^{\delta} C^\delta  \cdot \nabla 
 u_\lambda^{\delta}\big) dxdy  
 + \iint\limits_{D\times Y} {\mathfrak T}_{\delta}\big( (u_\lambda^{\delta})^2 k^\delta \big)  dxdy 
 \Big) \\
& =\lim_{\delta \to 0}  \Big(\int_D A^{\delta}(x)\nabla u^{\delta}_{\lambda}(x) \cdot \nabla u^{\delta}_{\lambda}(x)dx +
\int_D B^{\delta}(x) \cdot \nabla u^{\delta}_{\lambda}(x) \, u^{\delta}_{\lambda}(x)dx  \\
& \qquad + \int_D u_\lambda^{\delta}(x) C^\delta(x) \cdot  \nabla  u_\lambda^{\delta}(x) dx  
 + \int_D  u_\lambda^{\delta}(x)^2 k^\delta(x) dx \Big) \\
& = \lim_{\delta \to 0}  \form^{\delta}(u^{\delta}_{\lambda}, u^{\delta}_{\lambda}) \\
& = \form^0(u_0,u_0) \\
&  \  = \int_D A^{\sf eff}(x)\nabla u_0(x) \cdot \nabla u_0(x) dx +\int_D B^{\sf eff}(x) \cdot \nabla u_0(x) \, u_0(x)dx \\
& \qquad  +\int_D u_0(x)\, C^{\sf eff}(x) \cdot \nabla u_0(x) dx  + \int_D u_0(x)^2 k^{\sf eff}(x)dx  \\
& = \iint\limits_{D\times Y} A(x, y) \big(\nabla_x u_0+\nabla_y u_1\big) \cdot 
\big(\nabla_x u_0+\nabla_y u_1\big)dxdy  \\
& \qquad +  \iint\limits_{D\times Y} B(x, y) \cdot \big(\nabla_x u_0+\nabla_y u_1\big) u_0 dxdy \\
  & \qquad  + \iint\limits_{D\times Y} u_0(x) C(x, y)  \cdot \nabla_x u_0(x)dxdy  + \iint\limits_{D\times Y} 
 u_0(x)^2 k(x,y)dxdy.
\end{align*}
Summarizing to the above calculations, we find that 
$$
\lim_{\delta \to 0} \iint\limits_{D\times Y} \big|
{\mathfrak T}_{\delta}\big(\nabla u^{\delta}_{\lambda}\big) 
-\nabla_x u_0 - \nabla_y u_1\big|^2 dxdy  =0,
$$
as desired.
 \end{proof}

For $u,v \in H^1_0(D)$, 
\begin{align*}
\form^0(u,v) & :=\int_D A^{\sf eff}(x) \nabla u(x) \cdot \nabla v(x)dx + \int_D B^{\sf eff}(x) \cdot \nabla u(x) \, v(x)dx \\ 
& \qquad + \int_D u(x) \, C^{\sf eff}(x) \cdot \nabla v(x)dx 
  + \int_D u(x)v(x) k^{\sf eff}(x)dx 
\end{align*}
defines a bilinear form on $L^2(D)$. We next show that the pair $(\form^0, H_0^1(D))$ becomes 
a lower bounded closed form on $L^2(D)$ for some constant $\beta_0$.

\begin{prop} \label{prop:semi}
 In addition to the conditions in Theorem \ref{thm1},  we further assume that the limit of 
$\tilde{C}^\delta$ satisfies that $C\in (L^\infty(D; L^{p_0}(Y)))^d$. Then $(\form^0, H_0^1(D))$ is 
a regular lower bounded closed  form on $L^2(D)$.
\end{prop}

\begin{proof}
Observe that $\omega_i \in L^\infty(D;{\mathcal  H}_{\#}(Y))$ for $i=0,1,\ldots, d$, it follows from 
\eqref{corr2} that 
$$
\int_Y A(x,y) \big( \xi + \nabla_y \phi_{\xi}(x, y) \big) \cdot  \nabla_y \phi_{\xi}(x,y) dy=0, \quad {\rm a.e.}  \;\;  x\in D, 
$$
where $\phi_{\xi}(x, y)=\sum_{j=1}^d \xi_j \omega_j(x, y)$ for $\xi \in \rd$.  So from {\bf (A2)$'' $} and Proposition 
\ref{prop:corr} we find that 
\begin{align*}
\sum_{i,j=1}^d a^{\sf eff}_{ij}(x) \xi_i \xi_j & = \int_Y \Big(\sum_{ij=1}^d a_{ij}(x,y)\xi_i \xi_j+ 
\sum_{i,j=1}^d \Big( \sum_{\ell=1}^d a_{i\ell}(x, y) \frac{\partial \omega_j}{\partial y_\ell}(x, y) \Big)\xi_i  \xi_j \Big)dy   \\
& = \int_Y \Big(A(x, y) \xi \cdot \xi  + A(x, y) \xi \cdot \nabla_y \phi_{\xi}(x, y) \Big)dy  \\
&  = \int_Y A(x, y)\Big(\xi + \nabla_y \phi_{\xi}(x, y)\Big) \cdot \Big(\xi + \nabla_y \phi_{\xi}(x, y)\Big) dy \\
& \ge \alpha \int_Y \Big|\xi + \nabla_y \phi_{\xi}(x, y)\Big|^2dy \\ 
& = \alpha \Big( |\xi|^2 +2 \int_Y \xi \cdot \nabla_y \phi_{\xi}(x, y) dy
  + \int_Y \Big|\nabla_y \phi_{\xi}(x, y)\Big|^2dy \Big)\\ 
& = \alpha \Big( |\xi|^2 + \int_Y \Big|\nabla_y \phi_{\xi}(x, y)\Big|^2dy \Big) \\ 
& \ge \alpha |\xi|^2
\end{align*}
and 
\begin{align*}
\sum_{ij=1}^d a^{\sf eff}_{ij}(x) \xi_i \xi_j & \le \beta \int_Y  \Big(|\xi|^2 + \big| \nabla_y \phi_{\xi}(x, y)\big|^2\Big)dy\\
& = \beta  \int_Y  \Big(|\xi|^2 +  \sum_{i=1}^d \Big(\sum_{j=1}^d\xi_j \frac{\partial \omega_j}{\partial y_i}(x, y)\Big)^2 
 \Big)dy \\
 & \le  \beta \int_Y \Big(|\xi|^2 +  |\xi|^2 \sum_{i=1}^d \big|\nabla_y \omega_i(x, y)\big|^2 \Big) dy \\
 & =  \beta \Big(1 + \sum_{i=1}^d \big\| | \nabla_y \omega_i(x, \cdot)| \big\|_{L^2(Y)} \Big) |\xi|^2 \\
 & \le \beta (1+M) |\xi|^2
\end{align*}
for some constant $M>0$. 

On the other hand, since $B$ and $C$ belong to  $(L^{p_0}(D\times Y))^d$,  $k\in L^{p_0/2}(D\times Y)$ 
and $\omega_i \in L^\infty(D; {\mathcal H}_{\#}(Y))$,  the following inequality holds  by similar
estimates in the proof of  Lemma \ref{lem:est1}:  
 for any $\vareps>0$, there exists $\beta'=\beta'(\vareps)>0$ (depending  only on $\vareps>0$, the norms of 
 $B$ and $C$ in $(L^{p_0}(D\times Y))^d$,  $k$ in $L^{p_0/2}(D\times Y)$ and $\omega_i$ in 
$L^\infty(D; {\mathcal H}_{\#}(Y))$) such that
\[ 
\begin{split}
& \dis \max\Big\{ \Big| \int_D B^{\sf eff}(x) \cdot \nabla u(x) u(x)dx \Big|, \ 
 \Big| \int_D u(x) C^{\sf eff}(x)  \cdot \nabla u(x)dx \Big|,   \\ 
 &\qquad  \int_D u(x)^2 k^{\sf eff} (x) dx \Big\}   \le \vareps {\mathbb D}(u,u) + \beta' \|u\|_{L^2(D)}^2, \qquad  {\rm for} \ u \in H_0^1(D).
\end{split}
\]
Then taking $\beta'=\beta'(\vareps)$ for some $0< \vareps<\alpha/3$, we find that
\begin{align*}
\form^0_{\beta_0}  (u,u) & =\int_D A^{\sf eff}\nabla u(x)  \cdot \nabla u(x)dx  + \int_D B^{\sf eff} \cdot \nabla u(x) \, u(x)dx \\ 
& \quad + \int_D u(x)\, C^{\sf eff}  \cdot \nabla u(x) dx + \int_D u(x)^2 k^{\sf eff}(x)dx  + \beta' (u,u) \\ 
& \ge (\alpha- 3\vareps) \ {\mathbb D}(u,u)   \ge 0
\end{align*}
hold for any $u \in H^1_0(D)$.  Therefore for $\beta>\beta'$,  it follows that 
$\form^0_{\beta}(u,u)$ is comparable to the Dirichlet integral ${\mathbb D}(u,u)$ for $u \in H^1_0(D)$
 which means that $\form^0$ satisfies {\bf (B1)} and {\bf (B3)} with index $\beta'$ similar to Proposition \ref{DF}. 
{\bf (B2)} is proved similarly, and then we can conclude that the pair $(\form^0, H^1_0(D))$ is a regular 
lower bounded closed form on $L^2(D)$ with index $\beta'$. 
\end{proof}

We finish this paper by discussing the consequences of Theorem \ref{thm1} 
in terms of convergence of diffusion processes as follows.

\begin{cor} 
Let $\lambda>\beta_0 \vee \beta'$ and $f_{\delta} \in L^2(D)$,  where 
$\beta_0$ is the constant in Theorem \ref{delta-DF} and $\beta'$ is the constant appeared in the proof of Proposition \ref{prop:semi}. 
Then  for each $\delta>0$,  the solution $u_{\lambda}^\delta$  is the $L^2$-resolvent of $f$ with respect  to the 
lower bounded closed form $(\form^{\delta}, H^1_0(D))$ {\rm (i.e.,  $u_{\lambda}^{\delta}=G^{\delta}_{\lambda} f_{\delta}$)}.  
Moreover,  if $u_0$ is the limit of $u_{\lambda}^\delta$ in Theorem \ref{thm1}
and $f_{\delta}$ converges to $f$  strongly  in $L^{2}(D)$,  then 
$u_0$ is also the $L^2$-resolvent of $f$ with respect to the closed form $(\form^0, H^1_0(D))$ 
{\rm (i.e.,  $u_0=G^0_{\lambda} f$)} so  that  $G_{\lambda}^{\delta} f_{\delta}$ converges to $G_{\lambda}^0 f$ strongly 
in $L^2(D)$, or equivalently the $L^2$-semigroups $T^{\delta}_tf_{\delta}$ converges to $T^0_t f$ strongly in $L^2(D)$ as $\delta$ tends to $0$,   for any $t>0$.
\end{cor}

A similar property holds for the limit form. If {\bf (A3)} is also satisfied by $C^\delta(x)$ and $k^\delta(x)$ and the limit $C(x,y)$ of the unfolding $\tilde{C}^\delta(x,y)$ belongs to $(L^\infty(D; L^{p_0}(Y)))^d$,  
 we find that  $(\form^\delta, H^1_0(D))$ becomes a lower bounded semi-Dirichlet form on 
$L^2(D)$ for each $\delta>0$ and so is the limit form $(\form^0, H^1_0(D))$.  
As for  the semi-Dirichlet form property  for the limit form (i.e., the Markov property  for the limit semigroups),   since the semigroup $\{T_t^\delta\}_{t>0}$ 
of the resolvent $\{G^{\delta}_{\lambda}\}_{\lambda}$ is sub-Markov for each $\delta>0$, then so is the limit semigroup $\{T^0_t\}_{t>0}$ of 
$\{G_\lambda^0\}_{\lambda}$.  Therefore, the limit form $(\form^0, H^1_0(D))$ is indeed a regular 
lower bounded semi-Dirichlet form on $L^2(D)$.  Then, we can conclude the convergence of the semi-Dirichlet forms in the sense of 
Mosco \footnote{see \cite{H98} for the definition of the 
Mosco convergence of semi-Dirichlet forms (see also \cite{M94}).} and the convergence of the corresponding diffusion processes.

\begin{theorem}  Suppose that the conditions in Theorem \ref{thm1} are satisfied and the limit $C(x,y)$ of the unfolding of $C^\delta$, 
$\tilde{C}^\delta(x,y)={\mathfrak T}_\delta(C^\delta)(x,y)$,  satisfies that  $C\in (L^\infty(D; L^{p_0}(Y)))^d$ (see Proposition \ref{prop:semi}).  
Suppose further that {\bf (A3)} holds for $C^\delta(x)$ and $k^\delta(x)$.  Then $(\form^\delta, H^1_0(D))$  and 
$(\form^0, H^1_0(D))$ become lower bounded semi-Dirichlet forms on $L^2(D)$ and 
the sequence of semi-Dirichlet forms $(\form^\delta, H^1_0(D))$ converges to 
$(\form^0, H^1_0(D))$ on $L^2(D)$ in the sense of Mosco  and the diffusion processes ${\sf M}^{\delta}$ corresponding to $(\form^{\delta}, H^1_0(D))$ converges to 
the diffusion process ${\sf M}^0$ corresponding to $(\form^0, H^1_0(D))$ in the finite-dimensional distribution sense with some initial conditions. 
\end{theorem}

\begin{remark} We do not assume any continuity for $A(x)$ here to show the convergence. In contrast, 
one of the authors has assumed that the map $x \mapsto A(x)$ is continuous in \cite{TU22}  
since the $2$-scale convergence is used to show the convergence of the corresponding Dirichlet forms there.
The $2$-scale convergence is equivalent to the weak convergence of unfolding operators 
\cite[Proposition 1.19]{CDG18}. Therefore this may pose difficulties regarding the choice of the test function space for the convergence.
\end{remark}

\subsection*{Statements and Declarations} 
On behalf of all authors,  the corresponding author states that there is no conflict of interest.

\subsection*{Data Availability Statement}
The authors do not analyze or generate any datasets because this work proceeds with a theoretical and mathematical approach. The relevant materials can be obtained from the references below.

\subsection*{Acknowledgments} 
This work was supported by the Thammasat University Research Unit in Gait Analysis and Intelligent Technology (GaitTech). A crucial part of the research was conducted while A.S.  was visiting Kansai University; its hospitality is appreciated.  The authors acknowledge  Y\={u}ki Inagaki for fruitful discussions.
This work was also supported by JSPS KAKENHI Grant Numbers 19K03552,  19H00643, and 22K03340.

\end{document}